\documentclass[12pt]{amsart}
\usepackage{cite}
\usepackage{amsaddr}
\usepackage[includehead,includefoot,margin=25mm]{geometry}
\usepackage{graphicx}% Include figure files
\usepackage{dcolumn}% Align table columns on decimal point
\usepackage{bm}% bold math
\usepackage{color}
\usepackage[none]{hyphenat}
\usepackage{
enumitem}
\usepackage{amsfonts,amsmath,amssymb,amsthm,amscd}

\usepackage{mathrsfs}
\usepackage{marginnote}

\usepackage[colorlinks=true,backref]{hyperref}

\usepackage{mathabx}

\definecolor{SpringGreen}{rgb}{0.0, 1.0, 0.5}

 \newtheorem{thm}{Theorem}[section]
 \newtheorem{cor}[thm]{Corollary}
 \newtheorem{lem}[thm]{Lemma}
 \newtheorem{prop}[thm]{Proposition}
 
 \newtheorem{prp}[thm]{Proposition}
 \theoremstyle{definition}
 
 \newtheorem{dfn}[thm]{Definition}
 \theoremstyle{remark}
 \newtheorem{rem}[thm]{Remark}

 \numberwithin{equation}{section}
 \newtheorem{notation}[thm]{\textbf{Notation}}

\newcommand{\2}{``}

\DeclareMathOperator{\Ric}{Ric}
\DeclareMathOperator{\grad}{grad}

\DeclareMathOperator{\trace}{trace}
\DeclareMathOperator{\diver}{div}

\newcommand{\swp}{\emph{wp}}
\newcommand{\mwp}{\emph{mwp}}

\newcommand{\gK}{\emph{gK}}

 \newcommand{\vect}[1]{\boldsymbol{#1}}

\begin{document}

%\verb"20241226_Einstein_mwp_re-ordenado.tex"
\bigskip

%\title{Scalar Curvature of Generalized Kasner Semi-Riemannian Manifolds}
%\title[Einstein bcmwp of Generalized Kasner Type]{Einstein Base Conformal Multiply Warped Product Manifolds of Generalized Kasner Type}

\title[Einstein mwp and gK with multdim base]{Einstein multiply warped products and generalized Kasner manifolds with multidimensional base}

%----------Author 1
\author[Dobarro]{Fernando Dobarro}

\address[Dobarro]{
Instituto de Ciencias Polares, Ambiente y Recursos Naturales,
Universidad Nacional de Tierra del Fuego, Ant\'artida e Islas del Atl\'antico Sur,
9410 Ushuaia,
Fuegia Basket 251,
Argentina.\\ \texttt{fdobarro@untdf.edu.ar}
}
%\email[Dobarro]{fdobarro@untdf.edu.ar}
%
%\thanks{Orcid ID first author: 0000-0003-1031-0507 \\ Orcid ID second author: 0000-0002-5806-0319}
%
%----------Author 2
%
\author[Rey]{Carolina Rey}
\address[Rey]{Departamento de Matem\'atica,
Universidad T\'ecnica Federico Santa Mar\'ia,
2390123 Valpara\'iso,
Avenida España 1680,
Chile. \\ \texttt{carolina.reyr@usm.cl}}
%\email[Rey]{carolina.reyr@usm.cl}



%\date{\today}

% \author{Z. Yoshida,^1^1\footnote{email: yoshida@k.u-tokyo.ac.jp} P. J. Morrison^2^2\footnote{{email: morrison@physics.utexas.edu}} and F. Dobarro^3^3\footnote{email: fdob07@gmail.com}}
% \affiliation{
% ^1^1Graduate School of Frontier Sciences, The University of Tokyo,
% Chiba 277-8561, Japan
% \\
% ^2^2Department of Physics and Institute for Fusion Studies, University of Texas, Austin, Texas 78712-1060, USA
% \\
% ^3^3Dipartimento di Matematica e Informatica, Universit\`a degli Studi di Trieste,
%%Via Valerio 12/b,
% Trieste 34127, Italy
% }

% \date{\today}

\begin{abstract}
The purpose of this paper is to provide conditions for the existence or non existence of non trivial Einstein multiply warped products,  specially of generalised Kasner type; as well as to show estimates of the Einstein parameter that condition the existence of such metrics.
\end{abstract}

\maketitle

\tableofcontents

%%%%%%%%%%%%%%%%%%%%%%%%%%%%%%%%%%%%%%%%%%%%%%%%%%%%%%%%%%%%%%%%%%%%%%%%%%%%%%%%%%%%%%%%%%%%5

\section{Introduction}
\label{sec:introduction}

Since several decades, warped Riemannian and semi-Riemannian metrics play a relevant role in geometry and physics, and more recently, a fundamental role in information theory and statistics.

Given two semi-Riemannian manifolds $(B, g_B)$ and $(F, g_F)$ of dimensions $n$ and $k$ respectively, and a smooth function $b \in C^\infty_{>0}(B):=\{f:B \rightarrow (0,+\infty)\textrm{ such that } f \textrm{ is }C^\infty\}$, the (singly) warped product (for short \swp) $B \times_{b} F$ is the product manifold $B \times F$ with the metric tensor $g := g_B \oplus b^2g_{F}$.  Bishop and O'Neill first introduced this concept to create a class of complete Riemannian manifolds with negative sectional curvature \cite{Bishop-O'Neill1969}. As was mentioned in \cite{B-YChen2017}, before the Bishop and O'Neill article, the concept of \swp~ arose in the mathematical and physical literature as semi-reducible spaces in \cite{Kruchkovich1957}.

\noindent The \swp s have become increasingly important in geometry due to their ability to generate differentiable manifolds furnished with curvatures with particular properties, see among many others
\cite{Besse2008,Case2010,B-YChen2017,dobarro-lamidozo1987,Kim-Kim2003,Klingenberg1978,Lancini2018PHD,ONeill1983, Petersen2006,Rimoldi2010,Unal2000PHD}.

\medskip

The application of $\swp$s has proven to be invaluable in Theoretical Physics, particularly in general relativity, cosmological models, black holes, string theory, Kaluza-Klein theory, extra-dimensions theory, quantum-gravity, etc. Among the several warped semi-Riemannian metrics applied in these fields, we mention those of
Schwarzschild, Robertson-Walker, Friedmann-Robertson-Walker, Reissner-Nordstr\"om, Kasner space-times and Ba\~nados-Teitelboim-Zanelli, de Sitter and anti-de Sitter black hole solutions. See, for instance, \cite{B-YChen2011,B-YChen2017,Beem-Ehrlich-Easley1996,Hawking-Ellis1973,ONeill1983,Allison-Unal2003,Choi-Kim2004,Carot-da Costa 93,dobarro-unal2009,Frolov2001,Dumitru2011,Dumitru2016,Bhunia2022,Neves2023,Guler2022} and references therein.

As was mentioned above, recently a special type of warped metrics has gained relevance in information theory
and statistics, with a focus on their applications
in statistical inference and statistical learning. The base conformal warped products introduced in
\cite{Dobarro-Unal2008a,Dobarro-Unal2007} joint to multiply warped products were found to be related to the Mahalanobis distance, Rao-Fisher metrics and Sasakian statistical manifolds; see, for instance, \cite{Said-Bombrum-Berthomieu2019,Furuhata-Hasegawa-Okuyama-Sato2017,Bensadon2015,Bensadon2016,McLachla2004} and references therein.

Semi-Riemnnian metrics satisfying $\text{Ric} = \lambda g$ for some constant $\lambda$ are called \textit{Einstein metrics}, and manifolds admitting such metrics \textit{Einstein manifolds}, see \cite{Besse2008,ONeill1983,Petrov1969,Stephani-Kramer-Maccallum-Hoenselaers-Herlt2003,Deshmukh-2023}.
The question of the existence of a compact Einstein \swp~ with a non-constant warping function was raised by Besse \cite{Besse2008}.

\medskip

In \cite{Kim-Kim2003} D-S. Kim and Y. H. Kim proved that an Einstein \swp~ space of Riemannian manifolds with non-positive scalar curvature and compact base is simply a Riemannian product.
S. Kim \cite{KimS2006} provided examples of
compact Riemannian manifolds with positive scalar curvature that cannot be considered as the basis of any
non-trivial Einstein \swp~ of Riemannian manifolds.

\textit{Multiply warped product} (\mwp, for short) is a generalization of \swp s introduced by \"{U}nal \cite{Unal2000PHD}
in order to consider multifiber products and deal with the study of several semi-Riemannian metrics
of origin in physics. The $\mwp$s have been studied extensively as Einstein manifolds (see, e.g., \cite{Dobarro-Unal2005,B-YChen2011,B-YChen2017,Karaca2020,Pal-Kumar2019,Bhunia2022}). In  \cite{Dobarro-Unal2005},  Dobarro and \"{U}nal generalized the well-known Kasner space-times in relativity (see for instance \cite{Kasner21,Harvey1989,Harvey1990,Petrov1969,Misner-Wheeler-Thorne1973,Stephani-Kramer-Maccallum-Hoenselaers-Herlt2003}), to a family of \mwp s where the basis $B$ is an interval with metric $-dt^2$, the warping functions are powers of a function in $C^\infty_{>0}(B)$ and more general conditions than those considered for the classical Kasner. See also \cite{Frolov2001,Halpern2001,Kokarev1995} for other generalizations of the Kasner metrics.
In this paper, we expand this idea to $\mwp$ with a $n$-dimensional base manifold and call them {\it generalized Kasner semi-Riemannian manifold} (\gK, for short). Particularly interesting is the discussion arising in the Harvey article \cite{Harvey1990} about the precise definition of Kasner metrics, compared to \cite{Kasner21,Kasner1921a,Petrov1969}.
We would also remark that, in special cases, the Mahalanobis and Rao-Fisher metrics take the form of generalized Kasner metrics.

\medskip

In \cite{Dobarro-Unal2005}, the authors obtained the relationships between the curvature of a given \mwp~ and those of its base and its fibers. Based on them, here we analyze conditions for the existence or nonexistence of Einstein \mwp s and Einstein generalized Kasner manifolds. We deal in particular with the case where the dimension of the base is $\geq 2$.  Our approach is inspired by the works \cite{Kim-Kim2003} and \cite{KimS2006}, developed for singly \swp s.

\noindent In \cite{Pal-Kumar2019} the authors showed that it is not possible to construct nontrivial Einstein \mwp s with compact Riemannian base and negative scalar curvature, and applied their results to Generalized Robertson-Walker
space-time and Generalized Friedmann-Robertson-Walker space-time.

\noindent Although our results are close to those of Pal-Kumar, they differ, not only in the conclusions, but also in the fact that we require fewer hypotheses to obtain them, see our Remark \ref{rem:Pal-Kumar-results}.

 In addition to providing examples illustrating the absence of non-trivial Einstein \mwp s, both of general and generalised Kasner type, it is also our aim to show estimates and properties of the Einstein parameter that condition the existence of such metrics.

Considering the structure of the Ricci tensor of a \mwp ~obtained in \cite{Dobarro-Unal2005}, in Corollary \ref{cor: Einstein_m} we make explicit the equations that characterize an Einstein \mwp. The type of computations applied to obtain these expressions follows the line of those developed in \cite[Sections 2 and 3]{Dobarro-Unal2005}.
As a byproduct of this initial analysis, we show in Corollary \ref{cor: 2 Einstein_m} that, \emph{\2a singly warped product $B\times_{b_1} F_1$ of constant scalar curvature $\tau$ with Einstein fiber $F_1$" is Einstein if and only if
\begin{equation*}%\label{eq: 3 Einst-mwp-a}
  \Ric_{B}=\frac{\tau}{n + s_1} g_{B}+
			\frac{s_1}{b_1}{\rm H}_{B}^{b_1},			
\end{equation*}
where $(B,g_B)$ is a semi-Riemannian manifold, $\Ric_{B}$ is its Ricci tensor and ${\rm H}_{B}^{b_1}$ the $g_B$-Hessian tensor of the warping function $b_1$.
}

\noindent In \cite{dobarro-lamidozo1987} (see also \cite{CotiZelati-Dobarro-Musina1997}) the authors analyzed by means of order, variational and bifurcation methods, the existence of singly \swp s of constant scalar curvature $\tau$ when the basis is a compact Riemannian manifold. Applying Corollary \ref{cor: 2 Einstein_m} and the bounds obtained in Remark \ref{rem:dl} (see also Theorem \ref{thm:range of lambda}) we analyze under which conditions of the parameter $\tau$ the singly \swp s obtained in that article could be an Einstein manifold.

\medskip

A key aspect of the above results is the relationship between equation \eqref{Einst-mwp-c} in Corollary \ref{cor: Einstein_m} with $m=1$, and the difference between the hypothetical scalar curvature of the \swp~if it were Einstein, and its scalar curvature under the single information that \eqref{Einst-mwp-a} and \eqref{Einst-mwp-b} are verfied. In Theorem \ref{Theor: DR} we consider the analogous relation for a \mwp, and then we show the relevant difference between the cases of a single fiber and that of more than one fiber. See the consequences in Corollary \ref{cor: pre-Einstein_m}.

\medskip

Considering again Corollary \ref{cor: Einstein_m}, we seek to obtain some necessary and/or sufficient conditions that make a \mwp ~an Einstein manifold, in particular under the hypothesis that the basis is Riemannian and perhaps compact. As consequence, we reobtain results in \cite{Pal-Kumar2019}, improving some of them.

Subsequently, for a \mwp, we analyze the feasibility of obtaining, in some sense, equation \eqref{Einst-mwp-c}, from \eqref{Einst-mwp-a} and \eqref{Einst-mwp-b}, as Kim D-S. and Kim Y. H. do in \cite{Kim-Kim2003} for a single \swp. Our result improves on that obtained by \cite{Pal-Kumar2019} by showing that they impose an unnecessary hypothesis. But we show that what they claim, in our view, is not feasible, except in the case m=1. Our result is as follows.

\begin{thm}\label{Thm: PalKu}
	Let $\left(B^{n}, g_{B}\right)$ be a compact Riemannian manifold of dimension $n \geqq 2$. Suppose that $b_i$ are non-constant smooth functions on $B$ satisfying (\ref{Einst-mwp-a}) for a constant $\lambda \in \mathbb{R}$ and some positive integers $s_i$,  $i=1,\dots,m$.
	 Then $b_1, \dots, b_m$ satisfy
\footnote{Note that our convention for the sign of the Laplacian does not agree with that in \cite{Besse2008}, but it does with that in \cite{ONeill1983}; i.e. $\Delta f$  denote the Laplace-Beltrami operator of $f$ given by $\trace H^f$, where $H^f$ is the Hessian of $f$ (cfr. \eqref{eq:Lap-Bel}).}
\footnote{Whether $B$ is Riemannian or semi-Riemannian, $\|\grad_B b\|_B$ denotes $g_B(\grad_B b, \grad_B b)$.}
  \[
\sum_{k=1}^m\frac{s_k}{2 b_k^{2}}d\Big\{\sum_{j=1}^ms_j\frac{b_k}{b_j}g_{B}(\grad_{B}b_k,\grad_{B}b_j)
+ b_k\Delta_B b_k - \|\grad_B  b_k\|_B^{2}
+ \lambda b_k^2\Big\}=0.
\]
\end{thm}

\noindent Our proof follows the lines in \cite{Kim-Kim2003}, paying special attention to a number of fibers $m \ge 2$. Central tools for it are the divergence of the Ricci tensor and the Bianchi identities.

%\medskip

%================================================
In Section \ref{sec:non-exist}, motivated by the example of S. Kim \cite[Theorem 3]{KimS2006}, we state a result of non-existence of a non-trivial Einstein generalized Kasner multiply warped product with compact Riemannian basis, viz.

\begin{thm}\label{Thm-productBase}
   Let $(N_1, g_{N_1})$ and $(N_2, g_{N_2})$ be compact Riemannian manifolds with dimensions $r_1$ and $r_2$ and scalar curvatures $\tau_{N_1}$ and $\tau_{N_2}$, respectively, such that $\tau_{N_1}$ is a positive constant and the total scalar curvature of $N_2$ is non-positive, that is, $ \displaystyle \int_{N_2} \tau_{N_2} \leq 0$. Consider the product Riemannian manifold $B^n=N_1^{r_1} \times N_2^{r_2}$ of dimension $n=r_1+r_2,$ with metric $g_B=g_{N_1} \oplus g_{N_2}$
   and total scalar curvature $\sigma_B$. Then,
        \begin{enumerate}
      \item    $N_1 \times N_2$ cannot be the base of any trivial Einstein \gK ~manifold.
          %\item If $B^n \times_{\vect{\varphi}^{\vect{p}}} \vect{F}$ is an Einstein $\gK$ manifold with $\Ric=\lambda %g$ and $\lambda \in (-\infty,0] $, then $B^n \times_{\vect{\varphi}^{\vect{p}}} \vect{F}$ is a trivial %product manifold and $\sigma_B \leq 0$.
              %
          %\item If $\sigma_B > 0$, there is no Einstein \gK ~manifold with base $B$.
        \item There are no non-trivial Einstein \gK ~manifold with base $B$.
        \end{enumerate}

\medskip

   %Let $\varphi \in C^{\infty}_{>0}(B)$,  $\vect{p} \in \mathbb{R}^m$ and let $\vect{F}=F_1 \times \dots \times F_m$ be %any product of pseudo-Riemannian manifolds. If  $B^n \times_{\vect{\varphi}^{\vect{p}}} \vect{F}$ is a $\gK$ Einstein %manifold with $\Ric=\lambda g$ for some $\lambda\in \mathbb{R}$, then \textcolor{green}{$\lambda \leq 0$ and }$B^n %\times_{\vect{\varphi}^{\vect{p}}} \vect{F}$ is a trivial product manifold.
%
   %, where for each
   %$1\leq i \leq$  $m \:p_i=\dfrac{1}{2} \log \dfrac{\mu_i}{\lambda} (\log \varphi)^{-1}$
\end{thm}

While in the previous result we showed a non-admissible Riemannian manifold with positive total scalar curvature as the basis
of any Einstein \gK, in Theorem \ref{thm:range of lambda}, for a given compact Riemannian manifold $B$
we obtain the range of admissible proportionality constants between the Ricci tensor and the metric tensor
for a Einstein nontrivial \gK with basis $B$.

\medskip

The structure of this paper is as follows. In Section \ref{sec:multiply wp}, we provide an overview of several properties of multiple warped product manifolds, including the representation of the Ricci tensor and the scalar curvature. In Section \ref{sec:gK}, we introduce the concept of generalized Kasner semi-Riemannian manifolds and present the formulas for their Riemann and Ricci tensors and scalar curvature.
Sections \ref{sec:Einstein} and \ref{sec:non-exist} contain the central results of this article.
In Section \ref{sec:Einstein}, we examine the conditions under which a $\mwp$ can be considered an Einstein manifold and present the proofs of Theorem \ref{Theor: DR} and Theorem \ref{Thm: PalKu}.
Section \ref{sec:non-exist} is dedicated to the proof of Theorems \ref{Thm-productBase} and \ref{thm:range of lambda}, and to the establishment of several properties of a product Einstein manifold. Finally, in Section \ref{sec:conclusions}, we summarize our findings and explore potential future research directions.
\hbadness=99999

\bigskip

%%%%%%%%%%%%%%%%%%%%%%%%%%%%%%%%%%%%%%%%%%%%%%%%%%%%%

\section{Preliminaries about Multiply Warped Products}
\label{sec:multiply wp}
In this section, we review certain properties of multiple warped product manifolds. For a more exhaustive understanding, we refer to \cite{Unal2000PHD,Unal2000,Dobarro-Unal2005}. Most of the notation and terminology is taken from Bishop and O’Neill \cite{Bishop-O'Neill1969}.

Throughout this article, unless explicitly mentioned, we will consider semi-Riemannian manifolds (also called pseudo-Riemannian).

\begin{dfn}\label{def:mwp} Let $(B,g_B), (F_1,g_{F_1}), \dots,
(F_m,g_{F_m})$ be {\it pseudo-Riemannian} manifolds of dimensions $n, s_1, \dots, s_m$, respectively, and
let $\vect{b}= (b_1, \dots , b_m) \in (C^\infty _{>0}(B))^m$. The {\it multiply
warped product} (\mwp, for short) $B \times_{b_1} F_1 \times_{b_2}
\cdots \times_{b_m} F_m$
($B\times_{\vect{b}}\vect{F}$, for short, where $\vect{F}=F_1 \times \dots \times F_m$) is the {\it product manifold}
$B \times F_1 \times \cdots \times F_m$ furnished with the metric tensor $g=g_B \oplus
b_{1}^{2}g_{F_1} \oplus \cdots \oplus b_{m}^{2}g_{F_m}
= g_B \oplus {\bigoplus}_{i=1}^m b_i^2 g_{F_i}
$. More explicitly, $g$
is defined by
\begin{equation} \label{idwp}
g=\pi^{\ast}(g_B) \oplus (b_1 \circ \pi)^2
\sigma_{1}^{\ast}(g_{F_1}) \oplus \cdots \oplus (b_m \circ \pi)^2
\sigma_{m}^{\ast}(g_{F_m}).
\end{equation}
\noindent $(B,g_B)$ is called the \textit{base} of $B\times_{\vect{b}}\vect{F}$ and for
each $i$,  $b_i$ is called the $i$-th {\it warping function} and
$(F_i,g_{F_i})$ the $i$-th {\it fiber} of $B\times_{\vect{b}}\vect{F}$.
\end{dfn}

A multiple warped product $B\times_{\vect{b}}\vect{F}$ for which each warping function $b_i$ is constant can be considered as a product manifold $B\times\vect{\tilde{F}}$, where the fiber $\vect{\tilde{F}}$ is just $\vect{F}$ with the re-scaled metric ${\bigoplus}_{i=1}^m b_i^2 g_{F_i}$.

 From now on, we will use the definition and the sign convention
for the {\it curvature} in \cite[p.
16-25]{Beem-Ehrlich-Easley1996}, \cite[p. 3]{Aubin1982}, \cite[p. 196]{Lee}, \cite[p. 12]{Yano1970} or \cite[p. 239]{Spivak1999} (note
the difference with \cite[p. 89]{ONeill1983}, \cite[p. 24]{Besse2008})
 namely: for an arbitrary
$n$-dimensional pseudo-Riemannian manifold $(N,g_N)$, let $\nabla^N$ denotes the Levi-Civita connection on $N$. Then, the Riemann
curvature tensor is
\begin{equation*}
  \mathrm{R}_N(X,Y)Z:= [\nabla^N_X,\nabla^N_Y] Z - \nabla^N_{[X,Y]} Z
  =\nabla^N_X \nabla^N_Y Z - \nabla^N_Y \nabla^N_X Z -
  \nabla^N_{[X,Y]} Z,
\end{equation*}
where $X,Y,Z \in \mathfrak X(N)$ and $[\cdot,\cdot]$ is the Lie bracket; besides the Ricci tensor is
\begin{equation*}
  \mathrm{\Ric}_N(X,Y):= \sum_{j=1}^n g_N(E_j,E_j)g_N(R(E_j,Y)X ,E_j),
\end{equation*}
where $\{E_j\}_j$ is a "local" frame field on $(N,g_N)$.
%%%
%%%Furthermore, for each p \in N p \in N , the Ricci curvature tensor is
%%%given by
%%%\begin{equation*}\label{}
%%%  \Ric (X,Y)= \sum_{i=1} ^n g_N(E_i,E_i) g_N(R(E_i,Y)X,E_i),
%%%\end{equation*}
%%%where \{E_1, \cdots , E_n\}\{E_1, \cdots , E_n\} is an orthonormal basis for T_p NT_p N.
%%%

\noindent Furthermore, on $(N,g_N)$, we will recall that given a function $\varphi \in C^\infty(N)$ its Hessian is the symmetric two covariant tensor
\begin{equation*}
  H_N^\varphi (X,Y) := \nabla^N \nabla^N \varphi (X,Y) = XY\varphi - (\nabla^N _X Y) \varphi = g_N(\nabla^N_X (\grad_{N} \varphi), Y)
\end{equation*}
and the Laplace-Beltrami operator of $\varphi$ is
\begin{equation}\label{eq:Lap-Bel}
  \Delta_{N} \varphi :={\nabla^N}^i {\nabla^N}_i \varphi =\frac{1}{\sqrt{|g_N|}}\partial_i(\sqrt{|g_N|}{g_N}^{ij} \partial_j \varphi ),
\end{equation}
where $\nabla^N$ is the Levi-Civita connection of $(N,g_N)$ and $\grad_N \varphi$ is the field $g_N$-metrically equivalent to the differential of $\varphi$.

\smallskip

In the following two propositions, we will recall the expressions of the Ricci tensor and the scalar curvature for a \mwp ~ (see \cite{Dobarro-Unal2005}).

\begin{prp}\label{prp: Ricci mwp} \cite[p. 81]{Dobarro-Unal2005}
Let $B\times_{\vect{b}}\vect{F}$ be a \emph{\mwp} and $\nabla$ its
Levi-Civita connection.
Then, for any $X,Y \in \mathfrak L(B)$, $V  \in \mathfrak L(F_i)$
and $W \in \mathfrak L(F_j)$:
\footnote{$\mathfrak L$ means the set of fields lifted by the corresponding projections. Besides, we will denote with the same letter a field and its lifted by the corresponding projection, e.g. if $X \in \mathfrak{X}(B)$ we will denote with $X$ its horizontal lift in $\mathfrak{L}(B) \subset \mathfrak{X}(B\times_{\vect{b}}\vect{F})$ too; similarly for the vertical lift $U \in \mathfrak{L}(F_i)$ of $U \in \mathfrak{X}(F_i)$.
%......O'Neill p. 25 and p. 89.
}
\begin{subequations}\label{eq:Ricci tensor mwp}
\begin{align}
&\hspace{-1.5cm}\Ric(X,Y)=\Ric_{B}(X,Y)-\sum_{i=1}^m
\frac{s_i}{b_i}{\rm H}_{B}^{b_i}(X,Y) \label{eq:Ricci tensor mwp a}\\
&\hspace{-1.5cm}\Ric(X,V)=0 \label{eq:Ricci tensor mwp b}\\
&\hspace{-1.5cm}\Ric(V,W)=0 \hbox{ if } i \neq j \label{eq:Ricci tensor mwp c}\\
&\hspace{-1.5cm}\Ric(V,W)=\Ric_{F_i}(V,W)
- \Bigg(\frac{\Delta_{B}b_i}{b_i}+(s_i-1) \frac{g_{B}(
\grad_{B}b_i, \grad_{B}b_i)}{b_{i}^2} \label{eq:Ricci tensor mwp d} \\
&  \qquad +\sum_{\begin{subarray}{l}k=1\\k\neq i\end{subarray}}^m%\sum_{k=1,k\neq i}^m
s_k \frac{g_{B}(\grad_{B}b_i,\grad_{B}b_k)}{b_i b_k}\Bigg)b_i^2
g_{F_i}(V,W) \hbox{ if } i=j, \nonumber
\end{align}
\end{subequations}
where $\Ric$ (respectively, $\Ric_B$ or $\Ric_{F_i}$) is the Ricci
curvature tensor of $B \times_{\vect{b}} \vect{F}$
(respectively, $(B,g_B)$ or $(F_i,g_{F_i})$).
\end{prp}

\begin{prp}\label{prp: scalar curvature mwp} \cite[p. 81]{Dobarro-Unal2005}
Let $B\times_{\vect{b}}\vect{F}$ be a \emph{\mwp} and $\nabla$ its
Levi-Civita connection. Then,
the scalar curvature $\tau$ of $B \times_{\vect{b}} \vect{F}$
admits the following expression
\begin{align}\label{eq:sc_in_du2005_p2.6_1}
\tau  = & \tau_{B}-2\sum_{i=1}^m s_i \frac{\Delta_{B}b_i}{b_i}
+\sum_{i=1}^m \frac{\tau_{F_i}}{b_{i}^2}-
\sum_{i=1}^m s_i(s_i-1)\frac{\| \grad_{B}b_i \|_{B}^{2}}{b_{i}^2}\\
& -  \sum_{i=1}^m
\sum_{\begin{subarray}{l}k=1\\k\neq i\end{subarray}}^m%\sum_{j=1,j\neq i}^m
%\sum_{\substack{k=1 , k \neq i }}^m
s_k s_i \frac{g_{B}(\grad_{B}b_i,\grad_{B}b_k)}{b_i b_k},\nonumber
\end{align}
where $\tau_B$ (respectively, $\tau_{F_i}$) is the scalar
curvature of $(B,g_B)$ (respectively, $(F_i,g_{F_i})$).
\end{prp}

%\pagebreak

\bigskip

\section{Generalized Kasner Semi-Riemannian Manifolds}\label{sec:gK}
In this section, we give an extension of Kasner space-times and consider their scalar
and Ricci curvatures.
\medskip
\begin{notation}
For $\vect{\alpha} = (\alpha_1, \dots, \alpha_m) \in \mathbb{R}^m$ and $\vect{r} = (r_1, \dots, r_m) \in \mathbb{R}^m$, we will denote  $\vect{\alpha}^{\vect{r}}=(\alpha_1^{r_1}, \dots, \alpha_m^{r_m}).$
\end{notation}
\medskip
\begin{dfn}\label{def:gKmwp} A \textit{generalized Kasner semi-Riemannian manifold} (\gK, for short) is a $B \times_{\vect{b}} \vect{F}$ \mwp, where
$\vect{b}=\vect{\varphi}^{\vect{p}}$ with $\vect{\varphi}_i = \varphi  \textrm{ for each } i \in \{1,\dots,m\}$, $\varphi \in C^{\infty}_{>0}(B)$ and $\vect{p} \in \mathbb{R}^m$.
\end{dfn}
From now on, for an arbitrary \gK  as in Definition \ref{def:gKmwp}, we will denote
\begin{equation}\label{eq:zeta-eta}\tag{$\zeta;\eta$}
  \zeta:= \sum_{l=1}^m s_l p_l \quad \text{ and } \quad \eta:= \sum_{l=1}^m s_l p_l^2. %\leqno(\zeta;\eta)
\end{equation}
We observe that: $\eta = 0$ if and only if $p_l = 0  \textrm{ for each } l\in \{1,2, \dots , m\}$. But $\zeta $ could be $0$, without all the $p_l$'s simultaneously being $0$.

Taking into account the following relations for the gradient, the Laplace-Beltrami operator and the Hessian, namely:
for all $q \in \mathbb{R}$ and all $\gamma \in C^\infty_{>0} (B)$,
\begin{equation}\label{eq:gradient}
  \grad_B \gamma^{q}= q \gamma^{q-1}\grad_B \gamma,
\end{equation}
\begin{equation}\label{eq:Lap2}
  \dfrac{\Delta_B \gamma^{q}}{\gamma^q}=q \left[\frac{\Delta_B \gamma}{\gamma}+(q-1)\frac{\Vert   \grad_B \gamma \Vert_B^2}{\gamma^2} \right],
\end{equation}
\begin{equation}\label{eq:Hessian2}
  \dfrac{1}{\gamma^q}H_B ^{\gamma^{q}}=q \left[\frac{1}{\gamma}H_B ^\gamma+(q-1)\frac{1}{\gamma^2} d\gamma \otimes d\gamma\right],
\end{equation}

Next, we recall the expressions for the Riemann tensor of \gK s.

\begin{cor} \label{cor:gK-Riemann curv tensor mwp} \cite[p. 80]{Dobarro-Unal2005}
Let $B\times_{\vect{\varphi^p}}\vect{F}$ be a \emph{\gK} and $\nabla$ its Levi-Civita connection.
Then, for any $X,Y,Z \in \mathfrak L(B)$, $V  \in \mathfrak L(F_i), W  \in \mathfrak L(F_j)$ and $U \in \mathfrak L(F_k)$:
\begin{subequations}\label{eq:0Rcurv tensor gK}
\begin{alignat}{2}
  &{\rm R}(X,Y)Z = {\rm R}_B(X,Y)Z,\label{eq:1Rcurv tensor gK a} \medskip\\
  &{\rm R}(V,X)Y = -\frac{{\rm H}_B^{\varphi^{p_i}}(X,Y)}{\varphi^{p_i}}V  =
    -p_i [(p_i - 1)\varphi^{- 2}d\varphi \otimes d\varphi (X,Y) + \varphi^{- 1} {\rm H}_B^{\varphi}(X,Y)]V,
    \label{eq:2Rcurv tensor gK b}\medskip\\
  &{\rm R}(X,V)W =\mathrm{R}(V,W)X=\mathrm{R}(V,X)W=0 \hbox{ if } i \neq j,
    \label{eq:3Rcurv tensor gK c}\medskip\\
  &{\rm R}(X,Y)V = 0,\label{eq:4Rcurv tensor gK d}\medskip\\
  &{\rm R}(V,W)X = 0  \hbox{ if } i=j,\label{eq:5Rcurv tensor gK e}\medskip\\
  &{\rm R}(V,W)U = 0  \hbox{ if } i=j \hbox{ and } i,j \neq k, \label{eq:6Rcurv tensor gK f}\medskip\\
  &{\rm R}(U,V)W = - p_i p_k \varphi^{2p_i - 2} g_{F_i}(V,W)
                     g_{B}(\grad_{B}\varphi,\grad_{B}\varphi)U \hbox{ if } i = j \neq k, \label{eq:7Rcurv tensor gK g}\medskip\\
  &{\rm R}(X,V)W = -\varphi^{p_i} g_{F_i}(V,W) p_i
     [(p_i -1)\varphi^{p_i -2} d \varphi (X) \grad_B \varphi + \varphi^{p_i -1} \nabla^{B}_{X}\grad_{B}\varphi ]
     \label{eq:8Rcurv tensor gK h}\medskip\\
    \hbox{ if } i=j,\nonumber\medskip\\
  &{\rm R}(V,W)U = \mathrm{R}_{F_i}(V,W)U + \label{eq:9Rcurv tensor gK i}\medskip\\
  & p_i^2 \varphi^{2p_i - 2}g_{B}(\grad_{B}\varphi, \grad_{B}\varphi)[g_{F_i}(V,U)W-g_{F_i}(W,U)V]
    \hbox{ if } i=j=k,\nonumber
\end{alignat}
\end{subequations}
where $\mathrm{R}$ (respectively, $\mathrm{R}_B$ or
$\mathrm{R}_{F_i}$) is the Riemann curvature tensor of $B\times_{\vect{\varphi^p}}\vect{F}$  (respectively, $(B,g_B)$ or
$(F_i,g_{F_i})$).
\end{cor}

%
%\footnote{The sign - in \eqref{eq:8Rcurv tensor gK h} is the correct because there is a typo in \cite[Proposition 2.4 %(8)]{Dobarro-Unal2005}.}
%

%\newpage

As a consequence, we have the following expressions for the Ricci tensor.

\begin{cor}\label{cor: Ricci gK_mwp}
	Let $(M,g)$ be the \emph{\gK}  $ B \times_{{\vect{\varphi}}^{\vect{p}}} \vect{F}$.
	Then, for any $X,Y \in \mathfrak L(B), V  \in \mathfrak L(F_i)$
	and $W \in \mathfrak L(F_j)$:
\begin{subequations}\label{eq:Ricci tensor gK_mwp}
	\begin{align}
	&\Ric(X,Y)=\Ric_{B}(X,Y)-\mathcal{H}_{B}^\varphi=\Ric_{B}(X,Y)-(\eta-\zeta) \frac{1}{\varphi^{2}} d \varphi \otimes d \varphi (X,Y) -\zeta \frac{1}{\varphi} H_{B}^{\varphi} (X,Y), \label{eq:Ricci tensor gK_mwp a}\\
	&\hbox{In particular, if } \zeta, \eta \neq 0,  \hbox{ then } \label{eq:Ricci tensor gK_mwp b}  \\
	& \Ric(X,Y)=\Ric_{B}(X,Y)-\beta \frac{ {H}_B^{\varphi^{1/\alpha}}}{\varphi^{1/\alpha}}(X,Y), \nonumber \\
	&\Ric(X,V)=0, \label{eq:Ricci tensor gK_mwp c}\\
	&\Ric(V,W)=0 \hbox{ if } i \neq j, \label{eq:Ricci tensor gK_mwp d}\\
%
%			&\Ric(V,W)=\Ric_{F_i}(V,W)- p_i\Delta_B\varphi \ \varphi^{2p_i-1}
%			g_{F_i}(V,W)  -\label{eq:Ricci tensor gK_mwp e} \\
%			%
%			&  \qquad \left(p_i^2 + (\zeta-1)p_i-1\right) || \grad_{B}\varphi||_B^2\varphi^{2p_i-2}
%			g_{F_i}(V,W) \hbox{ if } i=j,\\
%
    &\Ric(V,W)=\Ric_{F_i}(V,W)- p_i\left(\frac{\Delta_B\varphi}{\varphi} +  (\zeta-1)
    \frac{\| \grad_{B}\varphi\|_B^2}{\varphi^{2}}\right)\varphi^{2p_i}
	g_{F_i}(V,W) \hbox{ if } i=j, \label{eq:Ricci tensor gK_mwp e}\\
    &\hbox{In particular, if } \zeta \neq 0,  \hbox{ then } \label{eq:Ricci tensor gK_mwp f}\\
    &\Ric(V,W) = \Ric_{F_i}(V,W) - \frac{p_i}{\zeta} \frac{\Delta_{B} \varphi^\zeta}{\varphi^\zeta} \varphi^{2p_i}
	g_{F_i}(V,W) \hbox{ if } i=j,\nonumber
	\end{align}
\end{subequations}
	where $\Ric$ (respectively, $\Ric_B$ or $\Ric_{F_i}$) is the Ricci curvature tensor of $B \times_{{\vect{\varphi}}^{\vect{p}}} \vect{F}$ (respectively, $(B,g_B)$ or $(F_i,g_{F_i})$), $\alpha=\dfrac{\zeta}{\eta}$ and $\beta=\dfrac{\zeta^{2}}{\eta}$, and $\mathcal{H}_{B}$ is the differential operator on $C_{>0}^{\infty}(B)$ defined by
\begin{equation*}
	\mathcal{H}_{B}^\phi:=\sum_{i=1}^{m} s_{i} \frac{H_{B}^{\phi^{p_{i}}}}{\phi^{p_{i}}}. %, \forall \phi \in C_{>0}^{\infty}(B).
\end{equation*}
\end{cor}

\bigskip

%%%%%%%%%%%%%%%%%%%%%%%%%%%%%%%%%%%%%%%%%%%%%%%%%%%%%%%%%%%%%%%%%%%%%%%%%%%%%%%%%%%%%%%%%%%%%%
Moreover, we get the following description of the scalar curvature for a $\gK$  manifold.

\begin{cor} \label{cor:ks-2} Let $(M,g)$ be the $\gK$  manifold $ B \times_{{\vect{\varphi}}^{\vect{p}}} \vect{F}$. Then,
its scalar curvature $\tau$ admits the following expression
\begin{equation} \label{eq:gK1}
\displaystyle{\tau=  -2 \zeta \frac{\Delta_B \varphi}{\varphi}
- %\textcolor{blue}{\mathbf{-}}
[\zeta^2 + \eta-2\zeta] \frac{\| \grad_{B}\varphi \|_{B}^{2}}
{\varphi^{2}} + \tau_B + \sum_{i=1}^{m}
\frac{\tau_{F_{i}}}{\varphi^{2p_{i}}}}.
\end{equation}
Furthermore, if $\zeta \neq 0$, taking $u=\varphi^{\frac{\zeta^2 + \eta}{2 \zeta}}$,
\begin{equation} \label{eq:gK2}
\displaystyle{-\dfrac{4}{1+\frac{\eta}{\zeta^2}} \Delta_B u + \tau_B u = \tau u - \sum_{i=1}^{m}
\tau_{F_{i}} u^{1-\frac{4}{1+\frac{\eta}{\zeta^2}}\frac{p_i}{\zeta}}}.
\end{equation}

\end{cor}

%\pagebreak

The next result is a generalization of \cite[Proposition 4.11]{Dobarro-Unal2005}.

\begin{prp} \label{prp:ks-2} Let $(M,g)$ be the \emph{\gK}  manifold $ B \times_{{\vect{\varphi}}^{\vect{p}}} \vect{F}$. Then it has
constant scalar curvature $\tau$ if and only if
\begin{enumerate}
\item each fiber $(F_i,g_{F_i})$ has constant scalar curvature
$\tau_{F_i}$ $\forall i \in \{1,\cdots,m\},$ and
\item $\tau $ satisfies either \eqref{eq:gK1}, or \eqref{eq:gK2} if $\zeta \neq 0$.
\end{enumerate}
\end{prp}

\bigskip

%%%%%%%%%%%%%%%%%%%%%%%%%%%%%%%%%%%%%%%%%%%%%%%%%%%%%%%%%%%%%%

\section{Einstein multiple warped manifolds}\label{sec:Einstein}
In this section, we study conditions for a $\mwp$ to be an Einstein manifold and prove Theorem \ref{Thm: PalKu}.

\medskip

From Proposition \ref{prp: Ricci mwp}, we have the following result. %(which is \cite[Corollary 2]{Pal-Kumar2019}).
\begin{cor}\label{cor: Einstein_m}
	The \emph{\mwp} $B\times_{\vect{b}}\vect{F}$ is Einstein with Ric $=\lambda g$ if and only if the next statements are all verified
	\begin{subequations}
		\begin{align}
			&\Ric_{B}=\lambda g_{B}+\sum_{i=1}^m
			\frac{s_i}{b_i}{\rm H}_{B}^{b_i};			\label{Einst-mwp-a}\\
			%
            %&\forall i \in \{1,\dots,m\}\textrm{ are verified  }   \nonumber\\
            %
			&\forall i \in \{1,\dots,m\}: \left(F_i, g_{F_i}\right) \text{ is Einstein with } \Ric_{F_i}=\mu_i g_{F_i}; \label{Einst-mwp-b}\\
            %
            %&\textrm{and} \nonumber \\
			%
			&\forall i \in \{1,\dots,m\}: \lambda b_i^2=\mu_i \!-\!\left(b_i\Delta_{B}b_i-
				\|\grad_{B}b_i\|_B^2  +\sum_{k=1}^m
			s_k\frac{b_i}{b_k}g_B\left( \grad_B b_i, \grad_B b_k  \right)\right),
			\label{Einst-mwp-c}
		\end{align}
	\end{subequations}
where $n=\dim B$ and $s_i=\dim F_i$.
\footnote{The sign before the summation in \eqref{Einst-mwp-a} is correct. Compare with \cite{Pal-Kumar2019}, where there is a typo.}

 Furthermore, for each $i$ \eqref{Einst-mwp-c} is equivalent to
\begin{equation}\label{eq:Einst-mwp-c-2}
  b_i^2\left[
  \sum_{k=1}^m
			s_kg_B\left( \frac{\grad_B b_i}{b_i}, \frac{\grad_B b_k}{b_k}  \right)
            +  \frac{\Delta_{B}b_i}{b_i} - \frac{\|\grad_{B}b_i\|_B^2}{b_i^2 } + \lambda \right]
            - \mu_i=0.
\end{equation}
\end{cor}

% {
% \tiny Other equivalent expressions for \eqref{Einst-mwp-c}
% \begin{equation*}
%   \sum_{k=1}^m s_k\frac{b_i}{b_k}g_B\left( \grad_B b_i, \grad_B b_k  \right) +
%   \lambda b_i^2 + b_i\Delta_{B}b_i - ||\grad_{B}b_i||_B^2 = \mu_i
% \end{equation*}
% %
% \begin{equation*}
%   b_i^2 \left[
%   \sum_{\begin{subarray}{l}k=1\end{subarray}}^m s_k
%   g_{B}\left(\frac{\grad_{B}b_i}{b_i},\frac{\grad_{B}b_k}{b_k}\right)
%  + \frac{\Delta_B b_i}{b_i}
%  -  \frac{||\grad_B  b_i||_B^{2}}{b_i^2}
% +  \lambda  \right]
% -\mu_i =0
% \end{equation*}
% %
% }

\bigskip

\subsection{Case of one single fiber}
\begin{rem}\label{rem:Einstein-singly-wp} Let a singly warped product $B\times_{\vect{b}}\vect{F}$, i.e. $m=1$, and let $\tau$ its scalar curvature. Let us assume $\lambda \in \mathbb{R}$. We will show
that
\noindent \2 if \eqref{Einst-mwp-b} is verified, then \eqref{Einst-mwp-a} implies
\begin{equation}\label{eq: pre-Einst-wp-c}
\lambda (n + s_1)  - \tau =  s_1 \dfrac{1}{b_1^2}\left[  b_1^2 \left(\dfrac{\Delta_B b_1}{b_1} + (s_1 -1)
\dfrac{\|\grad_B  b_1\|_B^{2}}{b_1^2} +\lambda\right)-   \mu_1\right],
\end{equation}
but the latter is equivalent to \eqref{Einst-mwp-c} if and only if $\tau=(n+s_1)\lambda$".

\noindent Note that \eqref{eq: pre-Einst-wp-c} expresses the difference between the hypothetical scalar curvature of $B\times_{\vect{b}}\vect{F}$, if it were Einstein, and the scalar curvature of $B\times_{\vect{b}}\vect{F}$ under the single information that \eqref{Einst-mwp-a} and \eqref{Einst-mwp-b} are verified.

\bigskip

\noindent Indeed, on one hand, by \eqref{eq:sc_in_du2005_p2.6_1}, we have
\begin{equation}\label{eq:sc_in_du2008_Ex.2.6}
   \tau = -2s_1 \dfrac{\Delta_B b_1}{b_1} - s_1(s_1 -1) \dfrac{\|\grad_B  b_1\|_B^{2}}{b_1^2} + \tau_B + \dfrac{\tau_{F_1}}{b_1^2}.
\end{equation}

\noindent On the other hand, taking the $g_B$-trace in \eqref{Einst-mwp-a}, we get
\begin{equation}\label{eq:tau_B-Einstein_F_1}
\tau_B = \lambda n + s_1 \dfrac{\Delta_B b_1}{b_1}.
\end{equation}

\noindent Then, by \eqref{eq:sc_in_du2008_Ex.2.6} and \eqref{Einst-mwp-b}, it holds

\begin{equation*}
\tau = -2s_1 \dfrac{\Delta_B b_1}{b_1} - s_1(s_1 -1) \dfrac{\|\grad_B  b_1\|_B^{2}}{b_1^2} + \lambda n + s_1 \dfrac{\Delta_B b_1}{b_1} + \dfrac{s_1 \mu_1}{b_1^2}.
\end{equation*}
Thus
\begin{equation*}
\lambda n  - \tau =  s_1 \left[  \dfrac{\Delta_B b_1}{b_1} + (s_1 -1) \dfrac{\|\grad_B  b_1\|_B^{2}}{b_1^2} - \dfrac{\mu_1}{b_1^2}\right],
\end{equation*}
or equivalently
%\begin{equation*}
%\lambda n  - \tau =  s_1 \dfrac{1}{b_1^2}\left[  b_1 \Delta_B b_1 + (s_1 -1) \|\grad_B  b_1\|_B^{2} - \mu_1\right]
%\end{equation*}
%or
\begin{align*}
\lambda (n + s_1)  - \tau &=  s_1 \dfrac{1}{b_1^2}\left[  b_1^2 \left(\dfrac{\Delta_B b_1}{b_1} + (s_1 -1)
\dfrac{\|\grad_B  b_1\|_B^{2}}{b_1^2} +\lambda\right)-   \mu_1\right]
%\\
%
%                         &= s_1 \dfrac{1}{b_1^2}\left[  b_1^2 \left(\frac{1}{s_1}\dfrac{\Delta b_1^{s_1}}{b_1^{s_1}}
%                          +\lambda \right) - \mu_1 \right]
.
\end{align*}

\noindent Therefore, \eqref{Einst-mwp-a} implies \eqref{eq: pre-Einst-wp-c}, which is equivalent to \eqref{Einst-mwp-c} if and only if $\tau =(n+s_1)\lambda$.
\footnote{
We observe that up to here, our proof requires $(B^n,g_B)$ to be only pseudo-Riemannian.\\
If we denote $\gamma := b_1^2 \left(\dfrac{\Delta_B b_1}{b_1} + (s_1 -1)
\dfrac{\|\grad_B  b_1\|_B^{2}}{b_1^2} +\lambda\right)$, then
\begin{align*}
\lambda(n+s_1) - \tau = -s_1 (\gamma -\mu_1)\frac{1}{b_1^2}.
\end{align*}
\\
Under the hypothesis that $(B^n,g_B)$ is compact and Riemannian,
in \cite[Proposition 5]{Kim-Kim2003} were proved, applying Bianchi identity,  that $\gamma $ is constant. As consequence, if $\mu_1 = \gamma$, then $\tau = \lambda (n+s_1)$ and then the \swp ~$B\times_{b_1}F_1$ results Einstein.\\
}
\end{rem}

\medskip

\begin{cor}\label{cor: 2 Einstein_m}
	The \emph{\swp} $B\times_{b_1} F_1$ is Einstein with Ric $=\lambda g$ if and only if
	\begin{subequations}
		\begin{align}
			&\Ric_{B}=\lambda g_{B}+
			\frac{s_1}{b_1}{\rm H}_{B}^{b_1};			\label{2 Einst-mwp-a}\\
			&\left(F_1, g_{F_1}\right) \text{ is Einstein with } \operatorname{Ric}_{F_1}=\mu_1 g_{F_1}; \label{2 Einst-mwp-b}\\
			&\tau =(n+s_1)\lambda,
			\label{2 Einst-mwp-c}
	\end{align}
	\end{subequations}
are all verified, where $n=\dim B$ and $s_1=\dim F_1$.

\medskip

\noindent In other words, \2a singly warped product $B\times_{b_1} F_1$ of constant scalar curvature $\tau$ with Einstein fiber $F_1$" results Einstein if and only if
\begin{equation}\label{eq: 3 Einst-mwp-a}
  \Ric_{B}=\frac{\tau}{n + s_1} g_{B}+
			\frac{s_1}{b_1}{\rm H}_{B}^{b_1}.			
\end{equation}
\end{cor}

\begin{rem}\label{rem:why must be positive}$~\!\!\!$[Why $\lambda$ and $\mu_1$ must be positive?] We observe that in Remark \ref{rem:Einstein-singly-wp}, \eqref{eq: pre-Einst-wp-c} could be written as
\begin{align*}
\lambda (n + s_1)  - \tau &= s_1 \dfrac{1}{b_1^2}\left[  b_1^2 \left(\frac{1}{s_1}\dfrac{\Delta_B b_1^{s_1}}{b_1^{s_1}}
                          +\lambda \right) - \mu_1 \right]\\
                          &= \dfrac{\Delta_B b_1^{s_1}}{b_1^{s_1}}+ s_1 \lambda - s_1 \mu_1 \dfrac{1}{b_1^2},
\end{align*}
or equivalently
\begin{align*}
-\Delta_B \psi -  s_1 \lambda \psi + s_1 \mu_1 \psi^{1 -\frac{2}{s_1}} = [\tau - \lambda (n + s_1)] \psi,
\end{align*}
where $0<\psi = b_1^{s_1}$. Thus, in Corollary \ref{cor: 2 Einstein_m}, \eqref{2 Einst-mwp-c} is equivalent to
\begin{align*}
-\Delta_B \psi -  s_1 \lambda \psi + s_1 \mu_1 \psi^{1 -\frac{2}{s_1}} = 0.
\end{align*}
\textcolor{black}{If $B$ is \emph{compact and Riemannian}, integrating over $B$ in the equation above yields that $\lambda$ and $\mu_1$ have the same sign. However, according to \cite[Theorem 1.1]{Kim-Kim2003}, if $b_1$ is not constant, then $\lambda$ must be positive, which implies that $\mu_1$ is also positive. Therefore, this remark enhances Corollary \ref{cor: 2 Einstein_m}.}
\end{rem}

\begin{rem}\label{divergence-thm} Now, we recall a classical result verified only for non null functions.
% that will be useful later.

\noindent For $f\in C^\infty_{>0} (N)$, with $(N,g_N)$ a \emph{compact Riemannian} manifold without boundary, it holds
\footnote{
%
%\textcolor[rgb]{1.00,0.00,0.00}{$\diver(u \grad v )= \delta (u dv) = -u \Delta v + g(\grad u,\grad v)$, where
%$\delta $ is the codifferential, \cite[p.11]{Yano1970}.}\\
%
%
$\diver(u \grad v )= \delta (u dv) = -u \Delta v - g(\grad u,\grad v)$, where
$\delta $ is the codifferential, \cite[p.184]{Gallot-Hulin-Lafontaine2004}, $\Delta$ defined in \eqref{eq:Lap-Bel} and $\diver X = \delta X^\flat $, with $X^\flat$ the dual 1-form of the vector field $X$ (i.e. $-$ the usual divergence in $\mathbb{R}^n$).
}
\[
 \operatorname{div}_{N}\left(\frac{\grad_N f}{f}\right)=
 %
 %\frac{- f \Delta f+|\grad_N f|^2}{f^2}=
 %
 -\frac{\Delta_N f}{f}+\left\|\frac{\grad_N f}{f}\right\|_N^2.
\]
And by the divergence theorem, we have
\[
0=\int_{N} \operatorname{div}_{N}\left(\frac{\grad_N f}{f}\right)\,dv_{g_N}=-\int_{N} \frac{\Delta_N f}{f}\,dv_{g_N}+\int_{N}\left\|\frac{\grad_N f}{f}\right\|_N^2\,dv_{g_N},
\]
so
\[
    \int_{N}\frac{\Delta_N f}{f} \,dv_{g_N}= \int_{N}\left\|\frac{\grad_N f}{f}\right\|_N^2 \,dv_{g_N}.
\]

%\noindent We would observe that this result is true only for non null functions.
%
\end{rem}

\bigskip

\begin{rem}\label{rem:dl} Let $B$ be a \emph{compact and Riemannian manifold}. Now, we observe that equation \eqref{2 Einst-mwp-c} in Corollary \ref{cor: 2 Einstein_m}, corresponds to the following equation in \cite{dobarro-lamidozo1987}
\begin{equation} \label{eq:wp-const-posit-sc}
-\frac{4 s_1}{s_1 +1} \Delta_B u + \tau_B u + s_1\mu_1 u^{\frac{s_1 -3 }{s_1 + 1}} = \lambda (n + s_1) u,
\end{equation}
where $0<u=b_1^{\frac{s_1 +1}{2}}$. In that paper, the authors look for positive solutions in a compact Riemannian manifold, to obtain warped products with constant scalar curvature, in this case $\lambda (n + s_1)$. By Remark \ref{rem:why must be positive}, \textcolor{black}{ in Corollary \ref{cor: 2 Einstein_m}} we are interested in the case $\mu_1 >0$ and $\lambda >0$. As consequence of the results in \cite{dobarro-lamidozo1987},
the positivity of $\mu_1$ implies that $\lambda (n + s_1) > \nu_0$ where the latter is the principal eigenvalue of the operator
$ -\dfrac{4 s_1}{s_1 +1} \Delta  + \tau_B $. So, $\lambda > \max \left\{0,\dfrac{\nu_0}{n+s_1}\right\}$.

\noindent On the other side, applying Remark \ref{divergence-thm} to equation \eqref{eq:tau_B-Einstein_F_1}, there results
$\displaystyle \lambda \leq \frac{1}{n|B|}\int_B \tau_B \,dv_{g_B}$.

\noindent Hence,
%and combining with Theorem \ref{thm:range of lambda} for one fiber $F_1$ and $p_1=1$ we obtain that
if the \swp~ $B\times_{b_1} F_1$ is Einstein with Ric $=\lambda g$, then
\begin{equation}
  \max \left\{0, \dfrac{\nu_0}{n+s_1}\right\} < \lambda \leq \frac{1}{n|B|}\int_B \tau_B \,dv_{g_B}.
\end{equation}
Note that if $s_1 \geq 3$,
\begin{equation*}
  \nu_0 = \inf \left\{ \int_B  \!\!\left ( \dfrac{4 s_1}{s_1 +1}\|\grad_B \psi \|_B^2 + \tau_B \psi^2 \right )\!\,dv_{g_B}; \psi \in H^1(B), \int_B \!\psi^2 \,dv_{g_B}=1  \right \}   \leq \nu_1,
  %\frac{1}{|B|}\int_B \!\tau_B \,dv_{g_B},
\end{equation*}

\noindent where
\begin{equation*}
    \nu_1 :=  \inf \left\{ \int_B  \!\!\left ( s_1\|\grad_B \psi \|_B^2 + \tau_B \psi^2 \right )\!\,dv_{g_B}; \psi \in H^1(B), \int_B \!\psi^2 \,dv_{g_B}=1  \right \} \leq \frac{1}{|B|}\int_B \!\tau_B \,dv_{g_B},
\end{equation*}
and by \eqref{eq:tau_B-Einstein_F_1} $\nu_1 = \lambda n$.  Furthermore, $\nu_0$ could be negative,  even if the total scalar curvature of $B$ is positive. So, if this were the case, the positive solutions of \eqref{eq:wp-const-posit-sc} obtained in \cite{dobarro-lamidozo1987} with $\lambda (n+s_1)> \nu_0$ and sufficiently close to $\nu_0$, would not correspond to Einstein manifolds. Consequently, these solutions would not satisfy the system \eqref{2 Einst-mwp-a} outlined in Corollary \ref{cor: 2 Einstein_m}.

\noindent Note that if the dimension of the fiber is $s_1=3$, then $\nu_0 = \nu_1$. The question remains whether $\nu_0 < \nu_1$ when $s_1 > 3$. Proving this would be an interesting problem to deal with.

\end{rem}

\begin{description}
  \item[Problem] For the example in Remark \ref{rem:positive-total-sc}, are $\nu_0$ or $\nu_1$ negative?
\end{description}

\bigskip

\subsection{Case of multiple fibers}
$ ~ $

Now, we consider an analogous problem,
%to the one discussed in Remark \ref{rem:Einstein-singly-wp},
but with at least two fibers.

\begin{thm}\label{Theor: DR}
Let $\left(B^{n}, g_{B}\right)$ be a pseudo-Riemannian manifold of dimension $n \geqq 2$; $m\geq 1$ functions $b_i \in C^\infty _{>0}$ satisfying \eqref{Einst-mwp-a} for a function $\lambda \in C^\infty (B)$ and positive integers $s_i$,  $i=1,\dots,m$. If for each $i \in \{1,2, \dots, m\}$ $F_i$ is a Einstein manifold of dimension $s_i$, i.e. $\Ric_{F_i} = \mu_i g_{F_i}$ with $\mu_i \in \mathbb{R}$, then
\begin{align}\label{eq:quasiEins}
\sum_{i=1}^m s_i \dfrac{1}{b_i^2} A_i
= \lambda \left(n +\sum_{i=1}^m s_i\right) - \tau ,
\end{align}
where
\begin{align}\label{eq:A_i}
&
A_i:=b_i^2 \left(
\sum_{\begin{subarray}{l}k=1\end{subarray}}^m
s_k  g_{B}\left(\dfrac{\grad_{B}b_i}{b_i},\frac{\grad_{B}b_k}{b_k}\right)+
\frac{\Delta_{B}b_i}{b_i} -
\frac{\| \grad_{B}b_i \|_{B}^{2}}{b_i^2}{\frac{}{}}
+ \lambda\right) -
\mu_i,
\end{align}
and
$\tau$ is the scalar curvature of the \emph{\mwp} ~ $B\times_{\vect{b}}\vect{F}$.
\end{thm}

\begin{rem}
Note that if $\lambda $ is constant, \eqref {eq:A_i} is the left-hand side of \eqref{eq:Einst-mwp-c-2}.
 Moreover, in Theorem \ref{Theor: DR}, if $m=1$, then \eqref{eq:quasiEins} coincides with \eqref{eq: pre-Einst-wp-c}.
\end{rem}

\begin{proof}[Proof of Theorem \ref{Theor: DR}]

By \eqref{eq:sc_in_du2005_p2.6_1}, we have
\begin{align}\label{eq:sc_in_du2005_p2.6_1 2 0}
\tau  = & \tau_{B}-2\sum_{i=1}^m s_i \frac{\Delta_{B}b_i}{b_i}
+\sum_{i=1}^m \frac{\tau_{F_i}}{b_{i}^2}-
\sum_{i=1}^m s_i(s_i-1)\frac{\| \grad_{B}b_i \|_{B}^{2}}{b_{i}^2}\\
& -  \sum_{i=1}^m
\sum_{\begin{subarray}{l}k=1\\k\neq i\end{subarray}}^m%\sum_{j=1,j\neq i}^m
%\sum_{\substack{k=1 , k \neq i }}^m
s_k s_i \frac{g_{B}(\grad_{B}b_i,\grad_{B}b_k)}{b_i b_k}.\nonumber
\end{align}

\noindent Taking the $g_B$-trace in \eqref{Einst-mwp-a}, it follows
\begin{equation}
\tau_B =  \lambda n + \sum_{i=1}^m
			{s_i}\dfrac{\Delta_{B}b_i}{b_i}
\end{equation}
Then, by \eqref{eq:sc_in_du2005_p2.6_1 2 0} and since $\Ric_{F_i} = \mu_i g_{F_i}$, there results

\begin{align*}
\tau - \lambda n  =&
-\sum_{i=1}^m s_i \frac{\Delta_{B}b_i}{b_i}
+\sum_{i=1}^m \frac{s_i \mu_i}{b_{i}^2}-
\sum_{i=1}^m s_i(s_i-1)\frac{\| \grad_{B}b_i \|_{B}^{2}}{b_{i}^2}\\
& -  \sum_{i=1}^m
\sum_{\begin{subarray}{l}k=1\\k\neq i\end{subarray}}^m%\sum_{j=1,j\neq i}^m
%\sum_{\substack{k=1 , k \neq i }}^m
s_k s_i \frac{g_{B}(\grad_{B}b_i,\grad_{B}b_k)}{b_i b_k},\nonumber
\end{align*}
or equivalently

\begin{align}\label{eq:sc_mwp_Einstein_fibers}
\tau - \lambda n  =&
-\sum_{i=1}^m s_i \frac{\Delta_{B}b_i}{b_i}
+\sum_{i=1}^m \frac{s_i \mu_i}{b_{i}^2}
+\sum_{i=1}^m s_i \frac{\| \grad_{B}b_i \|_{B}^{2}}{b_{i}^2}\\
& -  \sum_{i=1}^m
\sum_{\begin{subarray}{l}k=1\end{subarray}}^m%\sum_{j=1,j\neq i}^m
%\sum_{\substack{k=1 , k \neq i }}^m
s_k s_i \frac{g_{B}(\grad_{B}b_i,\grad_{B}b_k)}{b_i b_k},\nonumber
\end{align}
Finally, adding $-\lambda \sum_{i=1}^m s_i$ to both sides of \eqref{eq:sc_mwp_Einstein_fibers} and applying \eqref{eq:A_i}, we get \eqref{eq:quasiEins}.
%
%\begin{align*}
%&\sum_{i=1}^m s_i \dfrac{1}{b_i^2}\left[
%b_i^2 \left(
%\sum_{\begin{subarray}{l}k=1\end{subarray}}^m
%s_k  g_{B}\left(\dfrac{\grad_{B}b_i}{b_i},\frac{\grad_{B}b_k}{b_k}\right)+
%\frac{\Delta_{B}b_i}{b_i} -
%\frac{\| \grad_{B}b_i \|_{B}^{2}}{b_i^2}{\frac{}{}}
%+ \lambda\right) -
%\mu_i \right]\\
%&= \lambda \left(n +\sum_{i=1}^m s_i\right) - \tau .\nonumber
%
%\end{align*}

\end{proof}

The latter Theorem \ref{Theor: DR} implies that: \eqref{Einst-mwp-c} is sufficient for $\tau = \lambda \left(n +\sum_{i=1}^m s_i\right)$. And, if $\tau = \lambda \left( n +\sum_{i=1}^m s_i \right)$, then
\begin{equation}
\label{eq:weak-eins-mwp2}
\sum_{i=1}^m s_i \dfrac{1}{b_i^2}A_i
= 0
\end{equation}
\noindent Note the analogy with \eqref{eq: pre-Einst-wp-c}. But unlike $m=1$, for $m \geq 2$, it is not possible to conclude from \eqref{eq:weak-eins-mwp2}
that for each $i \in \{1, \dots, m\}$ $A_i$ is null, unless for instance, they all have the same sign. So, without any additional hypothesis, we cannot conclude \eqref{Einst-mwp-c}.
%\end{rem}
%
\begin{cor}\label{cor: pre-Einstein_m}
Let $B\times_{\vect{b}}\vect{F}$ be a \emph{\mwp} such that each fiber $F_i$ of dimension $s_i$ is Einstein with scalar curvature $\tau_{F_i}  = s_i \mu_i$, with $\mu_i$ constant. Let $\lambda$ be a function in $C^\infty (B)$ also. Hence, \eqref{Einst-mwp-a} is sufficient for \eqref{eq:quasiEins}. In addition, if $\tau = \lambda \left(n +\sum_{i=1}^m s_i\right)$, then \eqref{Einst-mwp-a} is sufficient for \eqref{eq:weak-eins-mwp2}.
On the other hand, $B\times_{\vect{b}}\vect{F}$ Einstein with $\Ric = \lambda g$ is sufficient for
\eqref{eq:weak-eins-mwp2}.
\end{cor}

\smallskip
In \cite{Kim-Kim2003} it was shown that if a compact base warped product $B \times_{b_1} F_1$ is Einstein with $\Ric=\lambda g $ and $\lambda\leq 0$, then the warped product is trivial (see \cite{Rimoldi2010,Case2010} for the case of a non-compact base). Next, we extend this result by considering the case where $\lambda>0$, and present an analogous result for $\mwp$.

\begin{lem}\label{lem: lambda estimates mwp}
    Let $(B,g_B)$ be a compact Riemannian manifold of dimension $n$, and $(F_i, g_{F_i})$ Einstein pseudo-Riemannian manifolds of dimension $s_i$ with $\Ric_{F_i}=\mu_i g_{F_i}$, and let $b_i \in C^\infty_{>0}(B)$, $i=1, 2, \dots, m$. Assume that the \emph{\mwp}~ $B\times_{\vect{b}}\vect{F}$ is an Einstein manifold with $\Ric=\lambda g$, with $\lambda$ constant and scalar curvature $\tau = (n + \sum _{i=1}^m s_i)\lambda$ also. Then, for each $i$, \2$\lambda = 0 \Leftrightarrow \mu_i=0$" and \2$\lambda > 0 \Leftrightarrow \mu_i > 0$".
    Furthermore,
    \begin{itemize}
        \item[\bf(i)] $\lambda = 0$ if and only if $\forall i: \mu_i=0$. 
        %\textcolor[rgb]{1.00,0.00,0.00}{Sospecho que en este caso son todas la variedades involucradas Ricci flat, si la curvatura escalar total de $B$ es cero}

  In addition, the total scalar curvature of $B$ is non-negative. And the total scalar curvature of $B$ is zero if and only if all the $b_i$s are constant, and so $B\times_{\vect{b}}\vect{F}$ is trivial and Ricci flat.

 In the case of one fiber ($m=1$), $b_1$ results constant, and so $B$ and $F_1$ are Ricci flat. As a consequence %$B\times_{\vect{b}}\vect{F}$
 $B\times_{b_1} F_1$ is trivial and Ricci flat.

 \smallskip

        \item[\bf(ii)] $\lambda <0 $ if and only if $\forall i: \mu_i <0$ and so $\forall i: b_i$ is constant.

        So, $B$ results Einstein with $\tau_B = \lambda n <0$, for each $i$ is $b_i^2=\dfrac{\mu_i}{\lambda}$ and $B\times_{\vect{b}}\vect{F}$ is trivial.

        \smallskip

    \item[\bf(iii)] $\lambda >0 $ if and only if $\forall i: \mu_i >0$. 
    
    Hence, $\forall i: 0<(\min_B b_i)^2 \le \dfrac{\mu_i}{\lambda} \le (\max_B b_i)^2$.
%        and \textcolor[rgb]{0.00,0.50,0.50}{$\displaystyle \lambda < \dfrac{1}{n|B|}\int_B \tau_B$.}
    \end{itemize}

    \noindent Furthermore,
\begin{equation}\label{eq: lambda estimate mwp}
  \lambda \leq \frac{1}{n |B|}\int_B \tau_B \,dv_{g_B}.
\end{equation}
This inequality is strict if at least one $b_i$ is non-constant. Moreover, the equality is true if and only if all the $b_i$s are constant.
\end{lem}

\begin{proof} Note that for each $i$, \eqref{Einst-mwp-c} (i.e. \eqref{eq:Einst-mwp-c-2}) is equivalent to
\begin{equation*}%\label{eq:Einst-mwp-c-2}
  0=\sum_{\begin{subarray}{l}k=1\\k\neq i\end{subarray}}^m
			s_kg_B\left( \frac{\grad_B b_i}{b_i}, \frac{\grad_B b_k}{b_k}  \right)
            +\lambda
            - \frac{\mu_i}{ b_i^2}
            +  \frac{\Delta_{B}b_i}{b_i} + (s_i - 1)\frac{\|\grad_{B}b_i\|_B^2}{b_i^2 },
\end{equation*}
or moreover to
\begin{equation*}%\label{pre-Einst-wp-c2}
0 = \sum_{\begin{subarray}{l}k=1\\k\neq i\end{subarray}}^m
			s_kg_B\left( \frac{\grad_B b_i}{b_i}, \frac{\grad_B b_k}{b_k}  \right) + \lambda - \mu_i \dfrac{1}{b_i^2}\ + \frac{1}{s_i}\dfrac{\Delta_B b_i^{s_i}}{b_i^{s_i}}
\end{equation*}
or
\begin{equation}\label{eq:lambda-mu mwp}
\lambda b_i^{s_i} = -b_i^{s_i} \sum_{\begin{subarray}{l}k=1\\k\neq i\end{subarray}}^m
			s_kg_B\left( \frac{\grad_B b_i}{b_i}, \frac{\grad_B b_k}{b_k}  \right)
             + \mu_i b_i^{s_i -2} - \frac{1}{s_i}\Delta_B b_i^{s_i}.
\end{equation}

\noindent Since $B$ is Riemannian, applying that the extremes are critical points and the maximum principle to \eqref{eq:lambda-mu mwp}, there results that if
\begin{equation}\label{eq:absolute extrems mwp}
0 < \min_B b_i = b_i({\widecheck{p}_i}) \le \max_B b_i = b_i(\widehat{p}_i),
\end{equation}
then
\begin{equation}\label{eq: max_principle-i mwp}
  \lambda b_i({\widecheck{p}_i})^2 \le \mu_i \le \lambda b_i({\widehat{p}_i})^2.
\end{equation}

\noindent Thus, the sign of $\lambda$ and $\mu_i$ are the same, and as a consequence, we get
\begin{description}
        \item[$\lambda = 0$ if and only if $\forall i: \mu_i=0$] Unlike the case $m=1$, \eqref{eq:lambda-mu mwp} does not imply that $b_i$ is constant. Indeed, the sign of $\displaystyle b_i^{s_i} \sum_{\begin{subarray}{l}k=1\\k\neq i\end{subarray}}^m
			s_kg_B\left( \frac{\grad_B b_i}{b_i}, \frac{\grad_B b_k}{b_k}  \right)$ is not well defined, and as a consequence neither the maximum principle nor the Bochner lemma may be applied.
        \item[$\lambda <0 $ if and only if $\forall i: \mu_i <0$] This comes from \eqref{eq: max_principle-i mwp} and as a consequence, it follows that $b_i({\widecheck{p}_i}) \ge b_i({\widehat{p}_i})$ and $b_i$ is constant for all $i$. Thus and \eqref{Einst-mwp-a}, $B$ results Einstein with $\tau_B = \lambda n <0$ and by \eqref{Einst-mwp-c} is $b_i^2=\dfrac{\mu_i}{\lambda} $ for each $i$.
        \item[$\lambda >0 $ if and only if $\forall i: \mu_i >0$] This comes from  \eqref{eq: max_principle-i mwp} and as a consequence, we get $\forall i: 0< b_i({\widecheck{p}_i})^2 \le \dfrac{\mu_i}{\lambda} \le b_i({\widehat{p}_i})^2.$
      \end{description}

In order to prove \eqref{eq: lambda estimate mwp}, we take the $g_B$-trace of \eqref{Einst-mwp-a}, integrate on $B$ and apply Remark \ref{divergence-thm}, so
\begin{equation}\label{eq:tau_B and lambda}
  \int_B \tau_{B} \,dv_{g_B}=\lambda n |B| + \sum_{i=1}^m
			s_i \int_B \frac{\Delta_B b_i}{b_i} \,dv_{g_B}= \lambda n |B| + \sum_{i=1}^m
			s_i \int_B \frac{\|\grad_B b_i\|_B^2}{b_i^2} \,dv_{g_B}.
\end{equation}
Note that the last term in the latter equation is $\geq 0$.

In the item i, as consequence of \eqref{eq:tau_B and lambda}, the total scalar curvature of $B$ results non negative. 
Moreover, since $\lambda$ is zero, 
the total scalar curvature of $B$ is zero if and only if all the $b_i$s are constant and $B\times_{\vect{b}}\vect{F}$ is trivial.

While the second part of the same is consequence of the Bochner lemma applied to \eqref{eq:lambda-mu mwp} and then of \eqref{Einst-mwp-a} and \eqref{Einst-mwp-c}.

\end{proof}

\bigskip

Now we prove Theorem \ref{Thm: PalKu}.
\begin{proof}[Proof of Theorem \ref{Thm: PalKu}]
	By taking the trace of both sides of (\ref{Einst-mwp-a}), we have	
	$$
	\tau_B=n \lambda + \sum_{k=1}^m\frac{s_k}{b_k} \Delta_B b_k,
	$$
	where $\tau_B$ denotes scalar curvature of $B$. Due to the second Bianchi identity, we know that
	$$
	d \tau_B=2 \diver_B(\Ric_B).
	$$
	Then, for all $ X \in \mathfrak X(B)$, we get
	\begin{equation}\label{divRic1}
	\diver_B \Ric_B(X)= \dfrac{1}{2} d \tau_B = \sum_{k=1}^m\frac{s_k}{2b_k^2} \left(  b_k d(\Delta_B b_k) -(\Delta_B b_k) d b_k\right)(X).
\end{equation}

	\noindent On the other hand, if $E_{1}, E_{2}, \cdots, E_{n}$ is a local orthonormal frame of $B$, then for any $k=1, \dots, m$ we have
	\begin{equation}\label{diver-frac}
	\diver_B\left(\frac{1}{b_k} H_B^{b_k}\right)(X)=\sum_{j=1}^n\left(\!D_{E_{j}}\!\left(\frac{1}{b_k} H_B^{b_k}\!\right)\!\right)\left(E_{j}, X\right)=-\frac{1}{b_k^{2}} H_B^{b_k}(\grad_B  b_k, X)+\frac{1}{b_k} \diver_B H_B^{b_k}(X),
	\end{equation}
	where, by definition,
	$$
	H_B^{b_k}(X, \grad_B  b_k)=\left(D_{X} d b_k\right)(\grad_B  b_k)=\frac{1}{2} d\left(\| \grad_B  b_k \|_B^{2}\right)(X).
	$$
	Then, equation \eqref{diver-frac} becomes
	$$
	\diver_B\left(\frac{1}{b_k} H_B^{b_k}\right)(X)=-\frac{1}{2 b_k^{2}} d\left(\|\grad_B  b_k \|_B^{2}\right)(X)+\frac{1}{b_k} \diver_B H_B^{b_k}(X).
	$$
	Hence, from this equation and (\ref{Einst-mwp-a}) it follows that
	\begin{equation}\label{diverRic}
	\diver_B\left(\Ric_{B}\right)=\sum_{k=1}^ms_k\diver_B\left(\frac{H_{B}^{b_k}}{b_k}\right)
	=\sum_{k=1}^m\frac{s_k}{2 b_k^{2}}\left\{- d\left(\|\grad_B  b_k \|_B^{2}\right)
	+2{b_k} \diver_B H_B^{b_k}\right\}.
	\end{equation}
	Applying the formula for the divergence of the Hessian tensor  (see for example \cite{Kim-Kim2003,Pal-Kumar2019}), \eqref{Einst-mwp-a} and the definition of Hessian, we get
	\begin{align}
\diver_B\left(H_B^{b_k}\right)(X) &=\Ric_B(\grad_B  b_k, X)  -\Delta_B(d b_k)(X)\nonumber\\
                                &=\lambda d b_k(X) + \sum_{j=1}^m\frac{s_j}{b_j}H_{B}^{b_j}(\grad_B  b_k, X)-\Delta_B(d b_k)(X).\nonumber
\end{align}
It follows from \eqref{diverRic} that

% {\tiny
% If $\varphi \in C^\infty (B)$, then   $ (d \delta + \delta d ) \varphi = \delta d \varphi = - \nabla^\nu \nabla_\nu \varphi = - \Delta \varphi$, see \cite[p. 28]{Aubin1982}.

% Furthermore $\Delta db_k = (d \delta + \delta d ) db_k = d (\delta d b_k)= d(-\Delta b_k)=-d\Delta b_k$
% }
  %
\begin{align}\label{diverRic2}
    \diver_B\left(\Ric_{B}\right)&=\sum_{k=1}^m\frac{s_k}{2 b_k^{2}}\Bigg\{
    - d\left(\|\grad_B  b_k \|_B^{2}\right)(X)\\
	&\left.+2\lambda{b_k} d b_k(X) +2b_k \sum_{j=1}^m \frac{s_j}{b_j}H_{B}^{b_j}(\grad_B b_k, X) + 2{b_k}d\Delta_B( b_k)(X)\right\}.\nonumber
\end{align}
Note that
\begin{equation*}
  	d\left(g_{B}(\grad_{B}b_k,\grad_{B}b_j)\right) (X)
	= H_B^{b_k}\left( \grad_{B}b_j, X\right) + H_B^{b_j}\left( \grad_{B}b_k,X\right),
\end{equation*}
and then
\begin{align*}
  	\sum_{k,j=1}^m\frac{s_k s_j}{b_k b_j}d\left(g_{B}(\grad_{B}b_k,\grad_{B}b_j)\right) (X)
	&= \sum_{k,j=1}^m\frac{s_k s_j}{b_kb_j}\left(
 H_B^{b_k}\left( \grad_{B}b_j, X\right) + H_B^{b_j}\left( \grad_{B}b_k,X\right) \right) \\
 &= 2\sum_{k,j=1}^m\frac{s_k s_j}{b_kb_j} H_B^{b_j}\left( \grad_{B}b_k,X\right).
\end{align*}
From \eqref{diverRic2} we obtain
\begin{align*}
    &\diver_B\left(\Ric_{B}\right)=\\
&=\sum_{k=1}^m\frac{s_k}{2 b_k^{2}}\left\{- d\left(\|\grad_B  b_k\|_B^{2}\right) +2 \lambda {b_k} db_k  + 2b_kd\Delta_B b_k\right\}
 +\sum_{j,k} \frac{s_k s_j}{2b_k b_j}d\left(g_{B}(\grad_{B}b_k,\grad_{B}b_j)\right).
\end{align*}
Equaling the last expression to \eqref{divRic1}, we have
$$
\sum_{k=1}^m\frac{s_k}{2 b_k^{2}}\left\{\sum_{j=1}^m s_j\frac{b_k}{b_j}d\left(g_{B}(\grad_{B}b_k,\grad_{B}b_j)\right) (X) \right\}
$$
$$
=\sum_{k=1}^m\frac{s_k}{2 b_k^{2}}\left\{-\Delta_B b_k d b_k + b_k d(\Delta_B b_k)+ d\left(\|\grad_B  b_k \|_B^{2}\right)
-2\lambda{b_k} d b_k(X) - 2{b_k}d\left(\Delta_B b_k\right)\right\}(X),
$$
and then
\begin{equation}\label{suma0}
0=\sum_{k=1}^m\frac{s_k}{2 b_k^{2}}\Big\{\sum_{j=1}^ms_j\frac{b_k}{b_j}d\left(g_{B}(\grad_{B}b_k,\grad_{B}b_j)\right) (X)
-d\left( -b_k\Delta_B b_k  + \|\grad_B  b_k \|_B^{2}
-\lambda b_k^2\right)(X)\Big\},
\end{equation}
Note that
\[
d \Bigg(\frac{b_k}{b_j}g_B\left( \grad_B b_k, \grad_B b_j  \right)\Bigg)
  = \frac{b_j db_k -b_k d b_j}{b_j^2} g_B\left( \grad_B b_k, \grad_B b_j  \right) + \frac{b_k}{b_j}d(g_B( \grad_B b_k, \grad_B b_j)).
\]

So,
\begin{align*}%\label{suma0}
&\sum_{k=1}^m\frac{s_k}{2 b_k^{2}}\sum_{j=1}^ms_j\frac{b_k}{b_j}d\left(g_{B}(\grad_{B}b_k,\grad_{B}b_j)\right)\\
&=\sum_{k=1}^m\frac{s_k}{2 b_k^{2}}\sum_{j=1}^ms_j  d \Bigg(\frac{b_k}{b_j}g_B\left( \grad_B b_k, \grad_B b_j  \right)\Bigg)
-\sum_{k=1}^m\frac{s_k}{2 b_k^{2}}\sum_{j=1}^ms_j\frac{b_j db_k -b_k d b_j}{b_j^2} g_B\left( \grad_B b_k, \grad_B b_j  \right)\\
&=\sum_{k=1}^m\frac{s_k}{2 b_k^{2}}d \Bigg(\sum_{j=1}^ms_j \frac{b_k}{b_j}g_B\left( \grad_B b_k, \grad_B b_j  \right)\Bigg)
-\dfrac{1}{2}\sum_{k=1}^m\sum_{j=1}^m \frac{s_k s_j}{ b_k^{2}b_j^{2}} (b_j db_k -b_k d b_j) g_B( \grad_B b_k, \grad_B b_j),
\end{align*}
but by simmetry, the last term in the right hand side of the latter is $0$.
%
%\footnote{
%From there, we get
%\[
%   \sum_{j,k} \frac{s_k s_j}{ b_k b_j} d \Bigg(\frac{b_k}{b_j}g_B\left( \grad_B b_k, \grad_B b_j  \right)\Bigg)
%  = \sum_{j,k} \frac{s_k s_j}{ b_k b_j}\Bigg( \frac{b_k}{b_j}dg_B\left( \grad_B b_k, \grad_B b_j  \right)\Bigg).
%\]
%}
%
Thus, \eqref{suma0} becomes
\begin{align}\label{eq:thm1.2}
0=\sum_{k=1}^m\frac{s_k}{2 b_k^{2}}d\Big\{\sum_{j=1}^ms_j\frac{b_k}{b_j}g_{B}(\grad_{B}b_k,\grad_{B}b_j)
+ b_k\Delta_B b_k - \|\grad_B  b_k \|_B^{2}
+ \lambda b_k^2\Big\} \left(X\right)
\end{align}
for all $X\in \mathfrak X(B)$, which concludes the proof.
\end{proof}

Moreover, we can see that even assuming \eqref{Einst-mwp-a} and \eqref{Einst-mwp-b}, we cannot get \eqref{Einst-mwp-c}.

\begin{rem}\label{rem:Pal-Kumar-results}
We observe that \eqref{eq:thm1.2} is equivalent to
\begin{align}\label{eq:thm1-2-bis}
0=\sum_{k=1}^m\frac{s_k}{2 b_k^{2}}
d\left\{
b_k^2 \left[\sum_{j=1}^m s_j g_{B}\left(\frac{\grad_{B}b_k}{b_k},\frac{\grad_{B}b_j}{b_j}\right)
+ \frac{\Delta_B b_k}{b_k} - \frac{\|\grad_B  b_k\|_B^{2}}{b_k^2}
+ \lambda
\right]
\right\} (X),
\end{align}
for all $X \in \mathfrak X(B)$.

Only in the case $m=1$ it is equivalent to
\begin{align}\label{rem:Pal-Kumar-results-m=1}
0=
d\left\{
b_1^2 \left[
(s_1 - 1)\dfrac{\|\grad_B  b_1\|_B^{2}}{b_1^2} + \dfrac{\Delta_B b_1}{b_1}
+ \lambda
\right]
\right\}.
\end{align}

The previous argument cannot be applied if the number of fibers is at least $2$. Indeed, if this were the case, we would not be able to deduce from \eqref{eq:thm1-2-bis} that each summand was $0$, i.e. that
for every $k$
\begin{align*}
d\left\{
b_k^2 \left[\sum_{j=1}^m s_j g_{B}\left(\frac{\grad_{B}b_k}{b_k},\frac{\grad_{B}b_j}{b_j}\right)
+ \frac{\Delta_B b_k}{b_k} - \frac{\|\grad_B  b_k\|_B ^{2}}{b_k^2}
+ \lambda
\right]
\right\}
\end{align*}
would be $0$.

Precisely, in our opinion, this is an incorrect conclusion in the work of Pal and Kumar \cite[Proposition 4]{Pal-Kumar2019}. Moreover, in that proposition, the authors require the unnecessary hypothesis \cite[(2) in Proposition 4]{Pal-Kumar2019}, as we show in the proof of our Theorem \ref{Thm: PalKu}.
\end{rem}

%%%%%%%%%%%%%%%%%%%%%%%%%%%%%%%%%%%%%%%%%%%%%%%%%%%%%%%%%%%%%%

\bigskip

\section{Einstein {\it gK} manifolds with compact Riemannian base}\label{sec:non-exist}

Throughout this section, $(B,g_B)$ is a compact Riemannian manifold of dimension $n$. For the sake of clarity, we begin this section by writing the Corollary \ref{cor: Einstein_m} for a \gK .

%==================================================

\begin{cor}\label{cor: Einstein gK_conditions}
	Let the \emph{\gK}  $ B \times_{{\vect{\varphi}}^{\vect{p}}} \vect{F}$ be. $ B \times_{{\vect{\varphi}}^{\vect{p}}} \vect{F}$ is Einstein with $\Ric =\lambda g$
 if and only if the next statements are all verified
\begin{subequations}\label{eq:Einstein conditions gK}
	\begin{align}
	&\lambda g_B=\Ric_{B}-\mathcal{H}_{B}^\varphi=\Ric_{B}-(\eta-\zeta) \frac{1}{\varphi^{2}} d \varphi \otimes d \varphi -\zeta \frac{1}{\varphi} H_{B}^{\varphi} , \label{eq:Einstein conditions gK a}\\
	&\hbox{In particular, if } \zeta, \eta \neq 0,  \hbox{ then }
\label{eq:Einstein tensor gK_mwp conditions gK b}\tag{\eqref{eq:Einstein conditions gK a}$(\zeta,\eta \neq 0)$}  \\
	& \lambda g_B=\Ric_{B}-\beta \frac{ {H}_B^{\varphi^{1/\alpha}}}{\varphi^{1/\alpha}}, \nonumber \\
	&\forall i \in \{1,\dots,m\}, \left(F_i, g_{F_i}\right) \text{ is Einstein with } \Ric_{F_i}=\mu_i g_{F_i}, \label{eq:Einstein conditions gK c}\\
    &\forall i \in \{1,\dots,m\}, \lambda \varphi^{2p_i}=\mu_i- p_i\left(\frac{\Delta_B\varphi}{\varphi} +  (\zeta-1)
    \frac{\| \grad_{B}\varphi\|_B^2}{\varphi^{2}}\right)\varphi^{2p_i}
	 , \label{eq:Einstein conditions gK e}\\
    &\hbox{In particular, if } \zeta \neq 0,  \hbox{ then }
    \label{eq:Einstein conditions gK f}\tag{\eqref{eq:Einstein conditions gK e}$(\zeta \neq 0)$}\\
    &\forall i \in \{1,\dots,m\}, \lambda \varphi^{2p_i} =\mu_i - \frac{p_i}{\zeta} \frac{\Delta_{B} \varphi^\zeta}{\varphi^\zeta} \varphi^{2p_i}
	,\nonumber
	\end{align}
\end{subequations}
	where $\Ric$ (respectively, $\Ric_B$ or $\Ric_{F_i}$) is the Ricci curvature tensor of $B \times_{{\vect{\varphi}}^{\vect{p}}} \vect{F}$ (respectively, $(B,g_B)$ or $(F_i,g_{F_i})$), $\alpha=\dfrac{\zeta}{\eta}$ and $\beta=\dfrac{\zeta^{2}}{\eta}$, and $\mathcal{H}_{B}$ is the differential operator on $C_{>0}^{\infty}(B)$ defined by
\begin{equation*}
	\mathcal{H}_{B}^\phi:=\sum_{i=1}^{m} s_{i} \frac{H_{B}^{\phi^{p_{i}}}}{\phi^{p_{i}}}. %, \forall \phi \in C_{>0}^{\infty}(B).
\end{equation*}
\end{cor}

%==================================================

\begin{rem}\label{rem:trivial gK}
Note that by Corollary \ref{cor: Einstein gK_conditions},
%\ref{cor: Ricci gK_mwp} (see also Corollary \ref{cor: Einstein_m}),
an Einstein $\gK$ is trivial (i.e. all the warping functions are constant) if and only if:
\begin{equation*}
  \lambda g_B = \Ric_B, \textrm{ so that $B$ is Einstein with }\lambda n=\tau_B
\end{equation*}
and for each $i=1, 2, \dots, m$,
\begin{equation*}
  \lambda \varphi_0^{2p_i} g_{F_i} = \mu_i g_{F_i}, \textrm{ so that } \lambda \varphi_0^{2p_i} = \mu_i,
\end{equation*}
where $\varphi_0$ is a positive constant.
If $\lambda = 0$, then $B$ and all fibers must be Ricci flat. However, $\lambda <0$ if and only if all the $\mu_i < 0$, and in such a case
$\displaystyle \varphi_0 = \left(\frac{\mu_i}{\lambda}\right)^{\frac{1}{2p_i} }$ with $\displaystyle \lambda = \frac{1}{|B|n}\int_B \tau_B \,dv_{g_B}$ and the scalar curvature of $ B \times_{{\vect{\varphi_0}}^{\vect{p}}} \vect{F}$ is $\tau = \lambda \left( n + \sum_{i=1}^m s_i \right)$, where $\displaystyle |B|=\int_B dv_{g_B}$. Analogously for $\lambda >0$.

This Remark, give rise to a strong condition for the $p_i$s in order to obtain a constant warping function that make our \gK a trivial Einstein manifold with $\lambda \neq 0$, namely
\begin{equation*}
  \varphi_0^{2p_i} = \dfrac{\mu_i}{\lambda}, \textrm{ for } i=1, \dots, m.
\end{equation*}

\end{rem}

\bigskip
%In this section we will prove Theorem \ref{Thm-productBase} by adapting arguments similar to those developed in %\cite[Theorem 3]{KimS2006}.

%\textcolor[rgb]{1.00,0.00,0.00}{CAMBIAR EL SIGNO DE LA DIVERGENCIA PARA QUE COINCIDA CON EL DEL LAPLACIANO O'NEILp84 %$\neq$ YANOp11}

\begin{rem}\label{rem:gK non positive lambda}
Let $(M,g)$ be the \emph{\gK}  $ B \times_{{\vect{\varphi}}^{\vect{p}}} \vect{F}$  Einstein manifold with $B$ compact, $\Ric = \lambda g$ and $\Ric_{F_i}=\mu_i g_{F_i}$ for $i=1,\dots,m$,  where $\lambda, \mu_i \in \mathbb{R}$.
Then, by \eqref{eq:Einstein conditions gK e}
%\eqref{eq:Ricci tensor gK_mwp e}
in Corollary \ref{cor: Einstein gK_conditions},
%\ref{cor: Ricci gK_mwp},
for each $i= 1, \dots, m$
\begin{equation}\label{eq:Einstein gK_mwp e}
  \lambda \varphi^{2p_i} - \mu_i =
			- p_i\left(\frac{\Delta_B\varphi}{\varphi} +  (\zeta-1)
            \frac{\| \grad_{B}\varphi\|_B^2}{\varphi^{2}}\right)\varphi^{2p_i}.
\end{equation}
Since $B$ is compact, there exist $\widecheck P$ and $\widehat P$ in $B$ such that,
\begin{equation*}
  \varphi(\widecheck P)=\min_B \varphi \textrm{ and } \varphi(\widehat P)=\max_B \varphi.
\end{equation*}
Hence, for each fixed $i$, applying the maximum principle and considering the several combinations of signs of the $p_i$'s and $\lambda$ in the latter equation, it results that $\lambda = 0$ if and only if $\mu_i=0$, and that $\lambda \leq 0$ is not admissible for $\varphi $ non-constant.

\noindent In particular, if $\lambda=\mu_i=0$ and $p_i\neq 0$, $\varphi$ constant is a necessary condition of equation \eqref{eq:Einstein gK_mwp e} and the Bochner lemma.

\noindent Taking in to account that a \gK is a type of \mwp, the precedent result is strongly close to Lemma \ref{lem: lambda estimates mwp}.

\noindent This remark leaves the interval $\displaystyle \left(0, \frac{1}{|B|n}\int_B \tau_B \,dv_{g_B}  \right)$ as the only admissible range of $\lambda$s for non-constant $\varphi$ and so indicates as a necessary condition that the total scalar curvature of $B$ is $>0$, see \eqref{eq: lambda estimate mwp}.

\bigskip

\end{rem}

\begin{proof}[Proof of Theorem \ref{Thm-productBase}]
\textcolor{black}{The item (1) is an inmediate consequence of the first line in Remark \ref{rem:trivial gK} and the fact that a standard product of manifolds is Einstein if and only if  any factor is Einstein with the same constant of proportionality of the whole product, see \cite[Proposition 1.99]{Besse2008} or Corollary
\ref{cor: Einstein gK_conditions}
%\ref{cor: Ricci gK_mwp}
 with $\varphi = 1$.
Indeed, if $N_1 \times N_2$ is the basis of a trivial \gK, Remark \ref{rem:trivial gK} implies that the standard product $N_1 \times N_2$ is Einstein, so $N_1$ and $N_2$ also, and all of them with the same constant $\lambda$. But the hipothesis of Theorem \ref{Thm-productBase} imply that $\lambda r_1 = \tau_{N_1} >0$ and $\lambda r_2 = \tau_{N_2} $ with $|N_2|\tau_{N_2} = \displaystyle \int_{N_2} \tau_{N_2} \leq 0 $. Hence, we obtain the contradiction $0 < \lambda \leq 0$.}

\textcolor{black}{In order to prove item (2), the Remark \ref{rem:gK non positive lambda}} allows us to reduce our proof to the case
$\lambda \in \displaystyle \left(0, \frac{1}{|B|n}\int_B \tau_B \,dv_{g_B}\right)$, where $B=N_1 \times N_2$ and $\varphi:N_1 \times N_2 \rightarrow (0,+\infty)$ is non constant.
\begin{description}
  \item[$\zeta \neq 0$ (so $\eta >0$)]
%Let us assume that $B^n \times_{\vect{\varphi}^{\vect{p}}} \vect{F}$ is a $\gK$ Einstein manifold with $\Ric=\lambda g$ %for some $\lambda\in \mathbb{R}$ and that $\zeta \neq 0$ and $\eta \neq 0$.
By Corollaries \ref{cor: Ricci gK_mwp}
and \ref{cor: Einstein gK_conditions}
%{and \ref{cor: Einstein_m}}
%\textcolor{red}{Proposition \ref{prp:ks-2}}
, we have
\begin{equation}\label{trazas}
 \Ric_{N_1}+\Ric_{N_2}=\Ric_{B}= \lambda g_B + \beta  \frac{ {H}_B^{\varphi^{1/\alpha}}}{\varphi^{1/\alpha}} \quad \text{ and } \quad
 \lambda \varphi^{2 p_i}=\mu_i-\frac{p_i}{\zeta} \frac{\Delta_B\varphi^{\zeta}}{\varphi^\zeta} \varphi^{2 p_i}.
\end{equation}
Now, by taking the $g_{N_1}$ and $g_{N_2}$ traces in the first equation respectively, we obtain the following:
\[
 \tau_{N_1}=\lambda r_1+\beta \frac{\Delta_{N_1} \varphi^{\frac{1}{\alpha}}}{\varphi^{\frac{1}{\alpha}}} \quad \text{and}\quad \tau_{N_2}=\lambda r_2+\beta \frac{\Delta_{N_2} \varphi^{\frac{1}{\alpha}}}{\varphi^{\frac{1}{\alpha}}}.
\]

Since by hypothesis $\tau_{N_1}$ and $\lambda$ are constant and $N_1$ is compact, then $\varphi^{\frac{1}{\alpha}}$ does not depend on $N_1$ because of Bochner's lemma (see for example \cite[Proposition 1.2, p.39]{Yano1970}). Then, since $N_2$ is compact also,  if $\lambda >0$, since $\beta >0$, by Remark \ref{divergence-thm}
\[
    \int_{N_2} \tau_{N_2} \,dv_{g_{N_2}}= \lambda r_2  |N_2| + \beta \int_{N_2}\frac{\Delta_{N_2} \varphi^{\frac{1}{\alpha}}}{\varphi^{\frac{1}{\alpha}}}\,dv_{g_{N_2}}=
    \lambda r_2  |N_2| + \beta \int_{N_2}\left\|\frac{\grad_{N_2} \varphi^{\frac{1}{\alpha}} }{\varphi^{\frac{1}{\alpha}}}\right\|_{N_2}^2 \,dv_{g_{N_2}} >0,
\]
where $|N_2|$ is the volume of $N_2$. Then,
we get a contradiction with the hypothesis of non-positive total scalar curvature of $N_2$.

  \item[$\zeta = 0$ and $\eta > 0$]

By
Corollaries \ref{cor: Ricci gK_mwp}
and \ref{cor: Einstein gK_conditions}
, we have
\begin{equation}\label{eq:trazas2-a}
 \Ric_{N_1}+\Ric_{N_2}=\Ric_{B}= \lambda g_B + \eta  \frac{1}{\varphi^{2}} d\varphi
\otimes d\varphi
\end{equation}
and
\begin{equation}\label{eq:trazas2-b}
 \lambda \varphi^{2 p_i}=\mu_i-{p_i} \varphi^{2 p_i} \left(\frac{\Delta_B\varphi}{\varphi}-\frac{\|\grad_B\varphi\|_B^2}{\varphi^2}\right)
\end{equation}
Now, by taking the $g_{N_1}$ and $g_{N_2}$ traces in \eqref{eq:trazas2-a} respectively, there results
\[
 \tau_{N_1}=\lambda r_1+\eta \frac{\|\grad_{N_1} \varphi\|_{N_1}^2}{\varphi^{2}} \textrm{ and }
 \tau_{N_2}=\lambda r_2+\eta \frac{\|\grad_{N_2} \varphi\|_{N_2}^2}{\varphi^{2}}.
\]
Then, since $N_2$ is compact and $\eta >0$, we have
\[
    \int_{N_2} \tau_{N_2} \,dv_{g_{N_2}} =
    \lambda r_2  |N_2| + \eta \int_{N_2}\left\|\frac{\grad_{N_2} \varphi }{\varphi}\right\|_{N_2}^2  \,dv_{g_{N_2}}>0,
\]
where $|N_2|$ is the volume of $N_2$. Then,
we get a contradiction with the hypothesis of non-positive total scalar curvature of $N_2$.
\end{description}
%
%Hence, the \gK ~$(N_1 \times N_2) \times_{\vect{\varphi}^{\vect{p}}} \vect{F}$ non trivial Einstein is not admissible.
%
Hence, the \gK ~$(N_1 \times N_2) \times_{\vect{\varphi}^{\vect{p}}} \vect{F}$ cannot be an Einstein manifold.

\end{proof}

\noindent Note that instead of the second equation of \eqref{trazas}, in the precedent proof we apply the more general
\eqref{eq:Einstein conditions gK e},
%\eqref{eq:Ricci tensor gK_mwp e}
precisely in the body of  Remark \ref{rem:gK non positive lambda}.
%(although the latter case is already covered by the general conclusion in \cite[Theorem5]{pal2019})

\bigskip

\begin{rem}\label{rem:positive-total-sc}$ ~ $
Here we observe that the total scalar curvature of the Riemannian compact base $B$ in Theorem \ref{Thm-productBase} can be take positive, indeed.

    \noindent Let $(N_1, g_{N_1})$ and $(N_2, g_{N_2})$ be compact Riemannian manifolds with dimensions $r_1$ and $r_2$ and scalar curvatures $\tau_{N_1}$ and $\tau_{N_2}$  respectively. Consider the product Riemannian manifold $B^n=N_1^{r_1} \times N_2^{r_2}$ of dimension $n=r_1+r_2,$ with metric $g_B=g_{N_1} \oplus g_{N_2}$ as before and denote by $\tau_B=\tau_{N_1}+\tau_{N_2} $ the scalar curvature of $B$.

\noindent Assuming that $\tau_{N_1}$ is constant, we have
$$
\begin{aligned}
\int_B \tau_B \,dv_{g_B}
&=|B| \tau_{N_1}+\int_B \tau_{N_2} \,dv_{g_B}
=\left|N_1\right|\left|N_2\right| \tau_{N_1}+\left|N_1\right| \int_{N_2} \tau_{N_2} \,dv_{g_{N_2}}\\
&=\left|N_1\right|\left(\left|N_2\right| \tau_{N_1}+ \int_{N_2} \tau_{N_2} \,dv_{g_{N_2}} \right)
 =\left|N_1\right|\left|N_2\right|\left(\tau_{N_1}+\frac{1}{\left|N_2\right|} \int_{N_2} \tau_{N_2} \,dv_{g_{N_2}}\right)\\
&=|B|\left(\tau_{N_1}+\frac{1}{\left|N_2\right|} \int_{N_2} \tau_{N_2}\,dv_{g_{N_2}}\right)
\end{aligned}
$$
where $|B|$ (respectively, $|N_1|,|N_2|$) denotes the volume of $B$ (respectively, $N_1,N_2$).
In addition, assuming that $\displaystyle \int_{N_2} \tau_{N_2} \leqslant 0$ as in the previous theorem and that $\tau_{N_1}$ is positive and large enough, then the total scalar curvature of $B$ results positive, i.e. $\sigma_B:= \displaystyle \int_B \tau_B \, dv_{g_B}>0$.
\end{rem}
\bigskip

\begin{rem}\label{rem:Einstein gK equations} $ ~ $ Now, if $(B^n,g_B)$ is a compact Riemannian manifold and $\varphi \in C^\infty_{>0}(B)$ verifies
\begin{equation}\label{eq:Einstein-gK}
  \Ric_{B}= \lambda g_B + \beta  \frac{ {H}_B^{\varphi^{1/\alpha}}}{\varphi^{1/\alpha}},
\end{equation}
where $\lambda, \beta, \alpha$ are real numbers and $\beta$ is positive,
 by taking the $g_B$-trace in the latter equation, we have
\begin{equation}\label{eq:S_B}
\tau_B=\lambda n+\beta \frac{\Delta_B \varphi^{\frac{1}{\alpha}}}{\varphi^{\frac{1}{\alpha}}},
\end{equation}
or equivalently
\begin{equation}\label{eq: eigenvalue potential S_B}
 -\beta \Delta_B \varphi^{\frac{1}{\alpha}}+\tau_B \varphi^{\frac{1}{\alpha}}=\lambda n \varphi^{\frac{1}{\alpha}},
\end{equation}
i.e.  $\lambda n$ is the main eigenvalue of the operator $-\beta \Delta_B+\tau_B$ and $\varphi^{\frac{1}{\alpha}}$ the corresponding positive eigenfunction.
Then, by the variational structure of this eigenvalue, namely:
\begin{equation}\label{eq:eigenvalue gK}
  \lambda n = \inf \left\{ \int_B (\beta \|\grad_B \psi \|^2 + \tau_B \psi^2) \, dv_{g_B}; \psi \in H^1(B), \int_B \psi^2 \, dv_{g_B}=1  \right \},
\end{equation}
where $H^1(B)$ is the Sobolev space of order $1$,
there results $\lambda n \leq \displaystyle \dfrac{1}{|B|}\int_B \tau_B \, dv_{g_B}$. Thus, if the total curvature of $B$ is non positive, then $\lambda n \leqslant 0$ (i.e. $\lambda \leq 0$).

%\textcolor{red}{Esto \'ultimo no se puede afirmar si $\int_B S_B>0$. }

\noindent Furthermore, integrating \eqref{eq:S_B} over $B$, by Remark \ref{divergence-thm}, we get
$$
\int_B \tau_B \, dv_{g_B}=\lambda n|B|+\beta \int_B\left\|\frac{\grad_B \varphi^{\frac{1}{\alpha}}}{\varphi^{\frac{1}{\alpha}}}\right\|^2 \, dv_{g_B}
$$
which implies that if $\varphi$ is non constant, then
$$
\int_B \tau_B \, dv_{g_B}-\lambda n|B|> 0.
$$

 On the other hand, if $(B^n,g_B)$ is a compact Riemannian manifold and $\varphi \in C^\infty_{>0}(B)$ verifies
\begin{equation}\label{eq:3.2a}
 \Ric_{B}= \lambda g_B + \eta  \frac{1}{\varphi^{2}} d\varphi
\otimes d\varphi,
\end{equation}
where $\eta$ is a positive constant.  By taking the $g_B$-trace in the latter equation, we have
\begin{equation*}
\tau_B=\lambda n+\eta \frac{\|\grad_B \varphi\|^2}{\varphi^{2}}
\textrm{ so } \int_B \tau_B \,dv_{g_B}-\lambda n|B| = \int_B \eta \frac{\|\grad_B \varphi\|^2}{\varphi^{2}} \,dv_{g_B} \geq 0,
%
%\quad \iff \quad -\eta |\nabla_B \varphi|^{2}+\tau_B \varphi^{2}=\lambda n \varphi^{2}
\end{equation*}
and the inequality is strict if $\varphi$ is non constant.
\end{rem}

As consequence of the latter remark, we have the following generalization to \gK s of the results obtained in Section 4 for Einstein singly warped products with compact Riemannian base.

\begin{thm} \label{thm:range of lambda} Let $(B^n,g_B)$ be a compact Riemannian manifold. If  $B^n \times_{\vect{\varphi}^{\vect{p}}} \vect{F}$ is a $\gK$ Einstein manifold with $\Ric=\lambda g$ for some $\lambda\in \mathbb{R}$, with $\varphi$ non constant
%and $\zeta \neq 0$ and $\eta \neq 0$
, then
\begin{equation}
 0 < \lambda <\frac{1}{n|B|}\int_B \tau_B \, dv_{g_B}.
\end{equation}
Furthermore, if $\zeta \neq 0$, then $\lambda n$ is the main eigenvalue of the operator $-\beta \Delta_B+\tau_B$ and $\varphi^{\frac{1}{\alpha}}$ the corresponding positive eigenfunction, with $\alpha, \beta$ as in the Corollary \ref{cor: Einstein gK_conditions}.
%\ref{cor: Ricci gK_mwp}.
%
% is given by \eqref{eq:eigenvalue gK} and $\varphi ^{\frac{1}{\alpha}} $ is the corresponding eigenfunction of \eqref{eq: %eigenvalue potential S_B}.
%
\end{thm}

\begin{proof}
It is sufficient to note that by Corollary \ref{cor: Einstein gK_conditions},
%(see also \ref{cor: Einstein gK_conditions}),
%\ref{cor: Ricci gK_mwp}
either \eqref{eq:Einstein-gK} or \eqref{eq:3.2a} are verified and apply Remark \ref{rem:gK non positive lambda}.
%that $\lambda >0$ by Lemma \ref{lem: lambda estimates mwp} or \cite[Theorem 5]{Pal-Kumar2019}.
\end{proof}

\begin{rem}
Note that in Theorem \ref{thm:range of lambda} we obtain a range for $\lambda$, while in Theorem \ref{Thm-productBase} we give an example where the total scalar curvature of the base is positive and there are no
%\textcolor{red}{non trivial}
Einstein \gK ~for that base.
\end{rem}

\begin{rem}\label{rem:Mustafa type} At this point we connect our Theorem \ref{thm:range of lambda} with a hypothesis in the lines of \cite{Mustafa2005}, namely.

\noindent Let $(B^n,g_B)$ be a compact Riemannian manifold. If  $B^n \times_{\vect{\varphi}^{\vect{p}}} \vect{F}$ is a $\gK$ Einstein manifold with $\Ric=\lambda g$ for some $\lambda\in \mathbb{R}$. Let us assume $\zeta \neq 0$.

\begin{enumerate}
  \item If $\tau_B \leq \lambda n$, then by \eqref{eq:S_B} and Bochner lemma $\varphi$ is constant, and as consequence $B^n \times_{\vect{\varphi}^{\vect{p}}} \vect{F}$ is trivial and $\tau_B = \lambda n$
  $\displaystyle \lambda \geq \frac{1}{n|B|}\int_B \tau_B \,dv_{g_B}$.

  \item If $\tau_B \geq \lambda n$, then by \eqref{eq:S_B} and Bochner lemma $\varphi$ is constant, and as consequence $B^n \times_{\vect{\varphi}^{\vect{p}}} \vect{F}$ is trivial and $\tau_B = \lambda n$
  $\displaystyle \lambda n \leq \frac{1}{|B|}\int_B \tau_B \,dv_{g_B}$, but the latter is true by \eqref{eq:eigenvalue gK}.
\end{enumerate}

\noindent In both cases, Remark \ref{rem:trivial gK} implies that $\displaystyle \lambda n = \frac{1}{|B|}\int_B \tau_B \,dv_{g_B}$.
We observe that Theorem \ref{thm:range of lambda} is in some sense a stronger result.

\noindent Combining both results,
if  $B^n \times_{\vect{\varphi}^{\vect{p}}} \vect{F}$ is a non trivial $\gK$ Einstein manifold with $\Ric=\lambda g$ for some $\lambda\in \mathbb{R}$, then $\min_B \tau_B  <\lambda n < \max_B \tau_B$ and
$\lambda n \in (\min_B \tau_B  , \max_B \tau_B)\cap(0,\sigma_B)$, where $\displaystyle\sigma_B:=\frac{1}{|B|}\int_B \tau_B \,dv_{g_B}$ $\leq \max_B \tau_B$.
\end{rem}

%%%%%%%%%%%%%%%%%%%%%%%%%%%%%%%%%%%%%%%%%%%%%%%%%%%%%%%%%%%%%%

\bigskip

\section{Conclusions and future developments}
\label{sec:conclusions}

In this work, we have provided a comprehensive set of conditions under which non-trivial Einstein multiply warped products, particularly those of generalized Kasner type, can exist. We have also derived estimates on the Einstein parameter that serve as key criteria for the existence of these metrics. In the following, we discuss the implications of our results and outline potential avenues for further research.

\medskip

\begin{rem}
Let us assume \eqref{Einst-mwp-a} and \eqref{Einst-mwp-b}, where $\lambda $ is a constant, then by Theorem \ref{Theor: DR}
\begin{subequations}
\begin{align}%\label{eq:tau modified}
&\lambda \left(n +\sum_{i=1}^m s_i\right) - \tau = \sum_{i=1}^m s_i \dfrac{1}{b_i^2} A_i = \label{eq:tau modified-a}\\
%
%%%& \sum_{i=1}^m s_i \dfrac{1}{b_i^2}\left[
%%%b_i^2 \left(
%%%\sum_{\begin{subarray}{l}k=1\end{subarray}}^m
%%%s_k  g_{B}\left(\dfrac{\grad_{B}b_i}{b_i},\frac{\grad_{B}b_k}{b_k}\right)+
%%%\frac{\Delta_{B}b_i}{b_i} -
%%%\frac{\| \grad_{B}b_i \|_{B}^{2}}{b_i^2}{\frac{}{}}
%%%+ \lambda\right) -
%%%\mu_i \right]= \label{eq:tau modified-b}\\
%
& \sum_{i=1}^m s_i \left[ \left(
\sum_{\begin{subarray}{l}k=1\end{subarray}}^m
s_k  g_{B}\left(\dfrac{\grad_{B}b_i}{b_i},\frac{\grad_{B}b_k}{b_k}\right)+
\frac{\Delta_{B}b_i}{b_i} -
\frac{\| \grad_{B}b_i \|_{B}^{2}}{b_i^2}{\frac{}{}}
+ \lambda\right) -
\mu_i \frac{1}{b_i^2} \right]
.\label{eq:tau modified-c}
\end{align}
\end{subequations}
So,
taking the differential of \eqref{eq:tau modified-a}/\eqref{eq:tau modified-c}, since $d \lambda =0$, we have
\begin{align}
d\tau= -
\sum_{i=1}^m s_i d\left[
\sum_{\begin{subarray}{l}k=1\end{subarray}}^m
s_k  g_{B}\left(\dfrac{\grad_{B}b_i}{b_i},\frac{\grad_{B}b_k}{b_k}\right)+
\frac{\Delta_{B}b_i}{b_i} -
\frac{\| \grad_{B}b_i \|_{B}^{2}}{b_i^2}{\frac{}{}}
 - \mu_i \frac{1}{b_i^2} \right].
\end{align}
\end{rem}

From this remark, we obtain some straightforward consequences, which are listed below:
\begin{itemize}[label=\checkmark]
\item $\tau$  is constant if and only if
$\sum_{i=1}^m s_i \dfrac{1}{b_i^2} A_i$ is constant.
  \item $\tau = \lambda \left( n + \sum_{i=1}^m s_i \right)$ if and only if $\sum_{i=1}^m s_i \dfrac{1}{b_i^2} A_i = 0$.
  \item if $m=1$, then \2$\tau = \lambda (n + s_1) $ if and only if the warped product $B \times_{b_1} F_1$ is Einstein", as a consequence of Corollary \ref{cor: 2 Einstein_m}.
\end{itemize}
In what follows, we will analyze deeper implications in some particular cases.

\subsection{Case \texorpdfstring{$m=1$}{Lg}, i.e. singly warped product}
  In this situation \eqref{eq:d-tau-m} takes the form
  \begin{align}\label{eq:d-tau-m=1}
&d \tau =
2  s_1 \dfrac{1}{b_1^3}\left[
b_1^2 \left(
 (s_1 - 1)
\frac{\| \grad_{B}b_1 \|_{B}^{2}}{b_1^2} +\frac{\Delta_{B}b_1}{b_1}
+ \lambda\right) -
\mu_1 \right]db_1,
\end{align}
but, by \eqref{rem:Pal-Kumar-results-m=1} we have
\begin{align}\label{rem:Pal-Kumar-results-m1-gK}
b_1^2 \left[
(s_1 - 1)\dfrac{|\grad_B  b_1|^{2}}{b_1^2} + \dfrac{\Delta b_1}{b_1}
+ \lambda
\right] = C \textrm{ constant}.
\end{align}
Hence, \eqref{eq:d-tau-m=1} takes the form
\begin{align}\label{eq:d-tau-m=1bis}
&d \tau =
2  s_1 \dfrac{1}{b_1^3}\left[
C -
\mu_1 \right]db_1,
\end{align}
and then

\begin{itemize}[label=\checkmark]
  \item $\tau $ is constant if and only if $C=\mu_1$ or $b_1$ is constant.
\end{itemize}

\subsection{\mwp ~with compact Riemannian base}
Now, let us assume \eqref{Einst-mwp-a} and \eqref{Einst-mwp-b}, where $\lambda $ is a constant, and consider $(B^n,g_B)$ a compact Riemannian manifold. Then, taking the differential of \eqref{eq:tau modified-a}
%/\eqref{eq:tau modified-b}
 on $(B^n,g_B)$ and multiplying by $-1$, we get

\begin{align*}
d \tau
&=\sum_{i=1}^m s_i 2\dfrac{1}{b_i^3}\left[
b_i^2 \left(
\sum_{\begin{subarray}{l}k=1\end{subarray}}^m
s_k  g_{B}\left(\dfrac{\grad_{B}b_i}{b_i},\frac{\grad_{B}b_k}{b_k}\right)+
\frac{\Delta_{B}b_i}{b_i} -
\frac{\| \grad_{B}b_i \|_{B}^{2}}{b_i^2}{\frac{}{}}
+ \lambda\right) -
\mu_i \right]db_i \\
& -\sum_{i=1}^m s_i \dfrac{1}{b_i^2}
d\left[
b_i^2 \left(
\sum_{\begin{subarray}{l}k=1\end{subarray}}^m
s_k  g_{B}\left(\dfrac{\grad_{B}b_i}{b_i},\frac{\grad_{B}b_k}{b_k}\right)+
\frac{\Delta_{B}b_i}{b_i} -
\frac{\| \grad_{B}b_i \|_{B}^{2}}{b_i^2}{\frac{}{}}
+ \lambda\right) \right] ,
\end{align*}
but the last addition is $0$ by \eqref{eq:thm1-2-bis}. Therefore, considering \eqref{eq:A_i}, there results
\begin{align}\label{eq:d-tau-m}
&d \tau =
2 \sum_{i=1}^m s_i \dfrac{1}{b_i^3}A_idb_i.
&
\end{align}
From this computations, we obtain
\begin{itemize}[label=\checkmark]
    \item $\tau$ is constant if and only if
$\sum_{i=1}^m s_i \dfrac{1}{b_i^3}A_idb_i=0$.
\end{itemize}

\subsection{Generalized Kasner manifold with compact Riemannian base}
 Let $(B^n,g_B)$ be a compact Riemannian manifold. Consider $b_i = \varphi^{p_i}$, where $\varphi \in C^\infty_{>0} (B)$ and let us assume \eqref{Einst-mwp-a} and
\eqref{Einst-mwp-b}. Then,  as shown previously, \eqref{eq:tau modified-a}/\eqref{eq:tau modified-c} is verified and takes the form
\begin{align}
\lambda \left(n +\sum_{i=1}^m s_i\right) - \tau &=
  \frac{\Delta_B \varphi^\zeta}{\varphi^\zeta}+
  \lambda \sum_{i=1}^m s_i
  - \sum_{i=1}^m s_i \mu_i \frac{1}{\varphi^{2p_i}},
  \nonumber
\end{align}
where $\zeta = \sum_{i=1}^m s_i p_i$.
Then, taking the differential in the latter, we get
\begin{align}\label{eq:dtau-gK}
d \tau =- d\left (\frac{\Delta_B \varphi^\zeta}{\varphi^\zeta}\right)+ \sum_{i=1}^m s_i \mu_i d(\varphi^{-2p_i}).
\end{align}
%This is a new, simpler expression equivalent to Equations \eqref{Einst-mwp-a} and \eqref{Einst-mwp-b}, in the case of a generalized Kasner multiply warped product.

On the other hand, in this context, \eqref{eq:thm1-2-bis} takes the form
\begin{align*}%\label{eq:thm1-2-bis}
0=\sum_{k=1}^m\frac{s_k}{2 b_k^{2}}
d\left\{
b_k^2 \left[\sum_{j=1}^m s_j g_{B}\left(\frac{\grad_{B}b_k}{b_k},\frac{\grad_{B}b_j}{b_j}\right)
+ \frac{\Delta b_k}{b_k} - \frac{|\grad_B  b_k|^{2}}{b_k^2}
+ \lambda
\right]
\right\} (X).
\end{align*}
\begin{align*}
0=\frac{1}{2}\sum_{k=1}^m s_k\frac{1}{\varphi^{2p_k}}
d
&\left\{
\varphi^{2p_k}
\left[\sum_{j=1}^m s_j p_k p_j \frac{\| \grad_{B}\varphi \|_{B}^{2}}{\varphi^2}
+p_k \left(\frac{\Delta_B \varphi}{\varphi}+(p_k-1)\frac{\Vert   \grad_B \varphi \Vert_B^2}{\varphi^2} \right)
\right.\right.\\
&\left.\left.
\qquad\quad- p_k^2\frac{\| \grad_{B}\varphi \|_{B}^{2}}{\varphi^2}
+ \lambda \right]
\right\}(X).
\end{align*}
Applying \eqref{eq:zeta-eta} and \eqref{eq:Lap2}, we obtain
\begin{align*}
0=\frac{1}{2}\sum_{k=1}^m s_k\frac{1}{\varphi^{2p_k}}
d\left\{
\varphi^{2p_k}p_k \left[(\zeta - 1)\frac{\Vert   \grad_B \varphi \Vert_B^2}{\varphi^2} + \frac{\Delta_B \varphi}{\varphi}\right]+\varphi^{2p_k} \lambda
\right\}
\end{align*}

\begin{align*}
0=\frac{1}{2}\sum_{k=1}^m s_k\frac{1}{\varphi^{2p_k}}
d\left\{
\varphi^{2p_k}p_k \frac{1}{\zeta}\frac{\Delta_B \varphi^\zeta}{\varphi^\zeta}+\varphi^{2p_k} \lambda
\right\}, \textrm{ if } \zeta\neq 0.
\end{align*}

\begin{align*}
0=\frac{1}{2}\sum_{k=1}^m s_k\frac{1}{\varphi^{2p_k}}
d\left\{
\varphi^{2p_k} \left(  p_k \frac{1}{\zeta}\frac{\Delta_B \varphi^\zeta}{\varphi^\zeta}+ \lambda \right)
\right\},
\end{align*}
which is equivalent, after some computations, to
% \begin{align*}
% \color{red}{0=\frac{1}{2}\sum_{k=1}^m s_k\frac{1}{\varphi^{2p_k}}   \left[2p_k \varphi^{2p_k -1}\left(  p_k \frac{1}{\zeta}\frac{\Delta_B \varphi^\zeta}{\varphi^\zeta}+ \lambda \right) d\varphi +
% %
% \varphi^{2p_k}  p_k \frac{1}{\zeta}d\left(\frac{\Delta_B \varphi^\zeta}{\varphi^\zeta} \right)
% \right]}
% \end{align*}

% \begin{align*}
% \color{red}{0=\sum_{k=1}^m s_k    2p_k \varphi^{-1}\left(  p_k \frac{1}{\zeta}\frac{\Delta_B \varphi^\zeta}{\varphi^\zeta}+ \lambda \right) d\varphi +
% %
% \sum_{k=1}^m s_k   p_k \frac{1}{\zeta}d\left(\frac{\Delta_B \varphi^\zeta}{\varphi^\zeta} \right)}
% \end{align*}

\begin{align}\label{eq:dtau-gK-laplacian}
0=2\frac{\eta}{\zeta} \frac{1}{\varphi} \frac{\Delta_B \varphi^\zeta}{\varphi^\zeta}d \varphi + 2 \zeta \frac{1}{\varphi}\lambda  d\varphi +
d\left(\frac{\Delta_B \varphi^\zeta}{\varphi^\zeta} \right).
\end{align}
% %
% \smallskip
% %
% where $\zeta = \sum_{k=1}^m s_k p_k$ and $\eta = \sum_{k=1}^m s_k p_k^2$.

\medskip

Hence, applying \eqref{eq:dtau-gK-laplacian} in \eqref{eq:dtau-gK}, there results
\begin{align} \label{eq:dtau-gK-final}
&d \tau =2\frac{\eta}{\zeta} \frac{1}{\varphi} \frac{\Delta_B \varphi^\zeta}{\varphi^\zeta}d \varphi + 2 \zeta \frac{1}{\varphi}\lambda  d\varphi- \sum_{i=1}^m s_i \mu_i 2p_i \varphi^{-2p_i -1}d \varphi.
\end{align}

\begin{align*} %\label{eq:dtau-gK-final}
&d \tau =\left[\frac{\eta}{\zeta}  \frac{\Delta_B \varphi^\zeta}{\varphi^\zeta} +  \zeta \lambda  - \sum_{i=1}^m s_i \mu_i p_i \varphi^{-2p_i }\right]2\frac{1}{\varphi}d\varphi.
\end{align*}

\begin{align*} %\label{eq:dtau-gK-final}
&d \tau =\zeta\left[\frac{\eta}{\zeta^2}  \frac{\Delta_B \varphi^\zeta}{\varphi^\zeta} + \lambda
- \frac{1}{\zeta}\sum_{i=1}^m s_i  p_i \mu_i \varphi^{-2p_i }\right]2\frac{1}{\varphi}d\varphi.
\end{align*}

%If $m \geq 2$
Since $\zeta \neq 0$, we conclude that
\begin{itemize}[label=\checkmark]
    \item $\tau $ is constant if and only if
\begin{align*}
&0=\frac{\eta}{\zeta^2}  \frac{\Delta_B \varphi^\zeta}{\varphi^\zeta} + \lambda
- \frac{1}{\zeta}\sum_{i=1}^m s_i  p_i \mu_i \varphi^{-2p_i } \textrm{ or }\varphi \textrm{ is constant}.
\end{align*}
\end{itemize}

% Note that, if $m=1$ and $p_1 = 1$, then the condition \eqref{eq:dtau-gK-final} coincides with the \eqref{rem:Pal-Kumar-results-m1-gK} obtained for a singly warped product. Indeed:

% \begin{description}
%   \item[\gK ~ with $m=1$] Here $\zeta = s_1 p_1$ and $\eta = s_1 p_1^2$, then

%   \begin{align*} %\label{eq:dtau-gK-final}
% &d \tau =\left[ \frac{1}{s_1} \frac{\Delta_B \varphi^{s_1 p_1}}{\varphi^{s_1 p_1}} +    \lambda  -   \mu_1  \varphi^{-2p_1 }\right]s_1 p_1 2\frac{1}{\varphi}d\varphi.
% \end{align*}

%  \begin{align*} %\label{eq:dtau-gK-final}
% &d \tau =\left(\varphi^{2p_1}\left[ \frac{1}{s_1} \frac{\Delta_B \varphi^{s_1 p_1}}{\varphi^{s_1 p_1}} +    \lambda \right]  -   \mu_1  \right)\varphi^{-2p_1 } s_1 p_1 2\frac{1}{\varphi}d\varphi.
% \end{align*}

% \begin{align*} %\label{eq:dtau-gK-final}
% &d \tau =\left(C  -   \mu_1  \right)\varphi^{-2p_1 } s_1 p_1 2\frac{1}{\varphi}d\varphi, \textrm{ by \eqref{rem:Pal-Kumar-results-m1-gK}}.
% \end{align*}

% \end{description}

% %.......................................................................................................................

% \textbf{Returning to a \gK with $m \geq 1$:}

Now, recall that the $b_i = \varphi^{p_i} $ must verify
\begin{subequations}
\begin{align}
			\Ric_{B}&=\lambda g_{B}+\sum_{i=1}^m
			\frac{s_i}{b_i}{\rm H}_{B}^{b_i} \label{hess-a}\\
                    &=\lambda g_{B}+\sum_{i=1}^m
			\frac{s_i}{\varphi^{p_i}}{\rm H}_{B}^{\varphi^{p_i}} \label{hess-b}\\
                    &=\lambda g_{B}+ \beta \frac{{\rm H}_{B}^{\varphi^{\frac{1}{\alpha}}}} {\varphi^{\frac{1}{\alpha}}},\label{hess-c}
			%\label{Einst-mwp-a}\\
\end{align}
\end{subequations}
where $\beta= \dfrac{\zeta^2}{\eta} \textrm{ and } \alpha= \dfrac{\zeta}{\eta}$, see \cite{dobarro-unal2008}.

\bigskip\bigskip

On the other hand, taking the $g_B$-trace in the latter equations
\begin{subequations}\label{eq:eigen-abc}
\begin{align}%\label{eq:eigen-abc}
			\tau_{B}&=\lambda n +\sum_{i=1}^m
			          \frac{s_i}{b_i}{\Delta}_{B}{b_i} \label{eigen-a}\\
	                &=\lambda n +\sum_{i=1}^m
			          s_i \frac{ \Delta_B \varphi^{p_i} }{\varphi^{p_i}} \label{eigen-b}\\
                    &=\lambda n + \beta \frac{\Delta_B \varphi^{\frac{1}{\alpha}}}{\varphi^{\frac{1}{\alpha}}},\label{eigen-c}
            %
		%\label{Einst-mwp-a}\\
\end{align}
\end{subequations}
for the same $\alpha$ and $\beta$.
Note that the equations \eqref{eq:eigen-abc} are strongly related, to the linear eigenvalue problem
\begin{align}\label{eq:principal eigenvalue}
-\beta \Delta_B \psi + \tau_B \psi = \nu \psi \textrm{ on } B.
\end{align}
Indeed, if $(B^n,g_B)$ is compact and Riemannian, we know that if $\nu_1$ is the principal eigenvalue of \eqref{eq:principal eigenvalue}, then all the corresponding eigenfunctions are multiples of one fixed eigenfunction, that is, the dimension of the eigenspace is one, and all the eigenfunctions do not change sign.
Let $\psi > 0$ be such an eigenfunction.

Thus, in our case,  $\lambda n$ must be $\nu_1$ and $\varphi^{\frac{1}{\alpha}} = t \psi$ for some $t \in ]0,+\infty[$.

If $\varphi = \psi^\alpha$ verify the Hessian condition \eqref{hess-c}, and consequently \eqref{eigen-c}, we  then it also needs to satisfy
 \eqref{eq:dtau-gK-final} in order to ensure that
$\tau$ is constant, i.e.
\begin{align}\label{eq:gKa}
\psi \Ric_B = \lambda \psi g_{B} + \beta {\rm H}_{B}^{\psi} \quad  \Rightarrow \quad \tau_B \psi = \underbrace{\lambda n}_{\nu_1} \psi + \Delta_B \psi
\end{align}
and
\begin{align*}
&0=\frac{\eta}{\zeta^2}  \frac{\Delta_B \varphi^\zeta}{\varphi^\zeta} + \lambda
- \frac{1}{\zeta}\sum_{i=1}^m s_i  p_i \mu_i \varphi^{-2p_i },
\end{align*}
where $\beta= \dfrac{\zeta^2}{\eta} \textrm{ and } \alpha= \dfrac{\zeta}{\eta}$. Hence, multiplying the last equation by $\beta$ and changing the variable, we obtain

\begin{align} \label{eq:gKb}
&0=  \frac{\Delta_B \psi^\beta}{\psi^\beta} + \lambda \beta
- \alpha\sum_{i=1}^m s_i  p_i \mu_i \psi^{-2p_i \alpha},
\end{align}
%
% \begin{align*} %\label{eq:gKb}
% &-\Delta_B \psi^\beta - \lambda \beta \psi^\beta =
% - \alpha\sum_{i=1}^m s_i  p_i \mu_i \psi^{\beta -2p_i \alpha},
% \end{align*}
%
or equivalently,
\begin{align*} %\label{eq:gKb}
&-\Delta_B \psi^\beta - \lambda \beta \psi^\beta =
- \alpha\sum_{i=1}^m s_i  p_i \mu_i \psi^{\beta(1-2p_i \frac{\alpha}{\beta})}.
\end{align*}

We are interested in pairs
 $(g_B,\psi)$ with $\psi >0$ that satisfy \eqref{eq:gKa} and \eqref{eq:gKb}, where the values $s_i$, $p_i$ and $\mu_i$ are given. Then, we would have a $\gK$  verifying \eqref{eq:gKa} with constant scalar curvature.
%
%In other words, if we could show that such a pair exists and satisfies these equations, then we would have a $\gK$ metric, which is Einstein, with warping functions given by powers of $\psi$.
 However, demonstrating this is beyond the scope of the current work and is left for future research.

\medskip

As a final question, we would like to study the following case in the context of a singly $\gK$ manifold:
\subsection{What happens if \texorpdfstring{$m=1$}{Lg} and \texorpdfstring{$p_1 \neq 0$}{Lg}?} In this case, $\zeta =s_1p_1 \neq 0$ and $\eta= s_1 p_1^2 >0$, then we have
\begin{equation*}
\beta = \frac{\zeta^2}{\eta}=\frac{s_1^2 p_1^2}{s_1 p_1^2}=s_1\quad \text{and }\quad \alpha= \frac{\zeta}{\eta}=\frac{s_1p_1}{s_1 p_1^2}=\frac{1}{p_1}.
\end{equation*}
Thus \eqref{eq:dtau-gK-final} becomes
\begin{align*} %\label{eq:dtau-gK-final}
&0=\frac{1}{s_1}  \frac{\Delta_B \varphi^{s_1p_1}}{\varphi^{s_1p_1}} + \lambda
-   \mu_1 \varphi^{-2p_1 }.
\end{align*}
Taking $\psi=\varphi^{\frac{1}{\alpha}}=\varphi^{p_1}$, we get
\begin{align*} %\label{eq:dtau-gK-final}
&0= \Delta_B \psi^{s_1} + \lambda s_1 \psi^{s_1}
- s_1  \mu_1 \psi^{s_1 -2 },
\end{align*}
so that
\begin{align} \label{eq:tau-constant-m=1}
&-\Delta_B \psi^{s_1} +(- \lambda s_1) \psi^{s_1} =
- s_1  \mu_1 (\psi^{s_1})^{1 -\frac{2}{s_1} }.
\end{align}
% \textcolor{red}{For clarity of exposition, we analyze first the case of one fiber and then the case of more than one fiber, showing that stronger results are obtained in the former.
% }
% \textcolor{red}{
% \begin{lem}\label{lem:lambda-mu1-conditions}
%     Let $(B,g_B)$ be a compact Riemannian manifold of dimension $n$, and $(F_1, g_{F_1})$  an Einstein  pseudo-Riemannian manifold of dimension $s_1$ with $\Ric_{F_1}=\mu_1 g_{F_1}$, and let $b_1 \in C^\infty_{>0}(B)$. Assume also that the warped product $B \times_{b_1} F_1$ is an Einstein manifold with $\Ric=\lambda g$, with $\lambda$ constant and scalar curvature $\tau = (n + s_1)\lambda$. Then, \2$\lambda = 0 \Leftrightarrow \mu_1=0$" and \2$\lambda > 0 \Leftrightarrow \mu_1 > 0$".
%     Furthermore,
%     %
%     \begin{itemize}
%         \item[\bf(i)] $\lambda = 0$ implies $\mu_1=0$ and $b_1$ constant, \textcolor{blue}{thus $B$ is Ricci flat.}% $\tau_B = 0$
%         \item[\bf(ii)] $\lambda <0 $ implies $\mu_1 <0$ and $b_1$ is constant,
%         \textcolor{blue}{thus $B$ is Einstein and $\tau_B  = \lambda n < 0$.}
%     \item[\bf(iii)] $\lambda >0 $ implies $\mu_1 >0$, $0<(\min_B b_1)^2 \le \dfrac{\mu_1}{\lambda} \le (\max_B b_1)^2$ and
%     if $b_1$ is non constant, then $\displaystyle \lambda n < \sigma_B  :=\dfrac{1}{|B|}\int_B \tau_B \,dv_{g_B}$.
%     \textcolor{blue}{Besides $-s_1 \Delta_B b_1 > -\tau_B b_1$}.
%     \end{itemize}
% \end{lem}
% }
In order to deal with Besse's conjecture, by using Lemma \ref{lem: lambda estimates mwp} and the non-existence result for non-trivial Einstein singly warped products with \(\lambda \leq 0\), we focus on the cases where \(\lambda > 0\) and \(\mu_1 > 0\). Therefore, we seek positive solutions to \eqref{eq:tau-constant-m=1} and \eqref{eq:gKa}, that is
 $$
\left\{ \begin{array}{rcl}
-\Delta_B \psi^{s_1} +(- \lambda s_1) \psi^{s_1} &=&
- s_1  \mu_1 (\psi^{s_1})^{1 -\frac{2}{s_1} }\\
\medskip
\psi \Ric_B &=& \lambda \psi g_{B} + \beta {\rm H}_{B}^{\psi} \quad \Rightarrow \quad \tau_B \psi = \underbrace{\lambda n}_{\nu_1} \psi + \Delta_B \psi.\\
\end{array} \right.
 $$
Note that \eqref{eq:tau-constant-m=1} is a nonlinear elliptic equation in $\psi^{s_1}$, which admits the constant
$\gamma = \left(\frac{\mu_1}{\lambda}\right)^{\frac{s_1}{2}}$ as positive solution.
If $\psi := \gamma $ were the unique positive solution of \eqref{eq:tau-constant-m=1} and $\psi$  were a solution of \eqref{eq:gKa}, then $(B^n,g_B)$ would be Einstein  manifold with $\Ric_B= \lambda g_B$. This leads us to analyse the following problem:

\medskip

\begin{description}
  \item[Problem] Let $(B,g_B)$ be a compact Riemannian manifold of dimension $n\geq 3$ without boundary with scalar curvature $\tau_B$ and total scalar curvature $\displaystyle \tau_0 = \frac{1}{|B|}\int_B \tau_B >0$ and let $ \omega \in (0,1)$ and $\mu>0$. Does $a\geq0$ exist such that if $\rho \in \left(a, \dfrac{\tau_0}{n}\right)$,
       the problem
      \begin{equation}\label{eq:conjecture}
        -\Delta_B u = \rho u - \mu u^\omega
      \end{equation}
      has a unique positive solution?
\end{description}

\medskip

To our knowledge, this remains an open problem in the field of geometric analysis. The number of positive solutions of equation \eqref{eq:conjecture} for $p>1$ with Dirichlet conditions on a bounded domain of $\mathbb{R}^n$ was studied by Ambrosetti in \cite{Ambrosetti-Mancini1979,Ambrosetti2011}. On the other hand, the problem the uniqueness of positive solution under Neumann conditions was studied by Ruf in \cite{Ruf} (Falta la cita). The author show that there is only one positive solution for $\lambda_1 < \rho <\lambda_2$, where $\lambda_1,\lambda_2$ are the eigenvalues of the Laplacian.No studies have been found regarding this equation within a Riemannian framework.

\medskip

The significance of this problem must not be overlooked. Suppose that the answer is affirmative. Then, given an Einstein warped product of the form $B \times_{\varphi^{p_1}} F_1$ with $a<\lambda <\dfrac{\tau_0}{n}$, $\psi = \varphi^{p_1}$ must satisfy both equations \eqref{eq:gKa} and \eqref{eq:tau-constant-m=1}.
Since the latter can only have constant positive solutions, $\varphi$ is a positive constant. Consequently, the warped product would be trivial, and the base is also an Einstein manifold. As a consequence, Besse's conjecture would be solved for $\rho \in \left(a,\dfrac{\tau_0}{n}\right)$.

The above reasoning proves the following proposition.
\begin{prop}
Let a \gK ~$B\times_{\vect{\varphi^p}}\vect{F}$ with compact base $(B,g_B)$ of dimension $n$ with $\zeta \neq 0$.
\begin{description}
  \item[i]
If $\Ric = \lambda g$ with $\lambda >0$, then $u:=\varphi^\zeta $ verifies
    \begin{align} \label{eq:condition}
        &-\Delta_B u - \lambda \beta u =
        - \alpha\sum_{i=1}^m s_i  p_i \mu_i u^{1-2p_i \frac{\alpha}{\beta}}.
    \end{align}
    \item[ii] If \eqref{eq:condition} has only one positive solution $u$, and if the \gK ~$B\times_{\vect{\varphi^p}}\vect{F}$
    verifies $\Ric = \lambda g$ with $\lambda >0$, then it holds $\varphi = u^{\frac{1}{\zeta}}$, that is, the \gK ~$B\times_{\vect{\varphi^p}}\vect{F}$  is a trivial Einstein manifold.
    Moreover, $\varphi$ is the $\varphi_0$ defined in Remark \ref{rem:trivial gK} and $\varphi^\zeta$ verifies \eqref{eq:condition}.
\end{description}
\end{prop}

\begin{rem}
    Notice that if $m=1$, then \eqref{eq:condition} is \eqref{eq:conjecture}, with $\rho = \lambda s_1 $ and $\omega = 1-\frac{2}{s_1}.$
\end{rem}

\end{document}